\documentclass[12pt,twoside,reqno]{amsart}
\usepackage{amsmath}
\usepackage{amsfonts}
\usepackage{amssymb}
\usepackage{color}
\usepackage{mathrsfs}
\usepackage{cite}
\usepackage{cleveref}
\usepackage{geometry}
\usepackage{marginnote}
\usepackage{todonotes}
\allowdisplaybreaks
\newcommand{\dom}{\mathrm{dom}}
\newcommand{\D}{D}
\newtheorem{theorem}{Theorem}[section]
\newtheorem{corollary}[theorem]{Corollary}
\newtheorem{lemma}[theorem]{Lemma}
\newtheorem{proposition}[theorem]{Proposition}

\newtheorem{remark}[theorem]{Remark}
\numberwithin{equation}{section}
\begin{document}
\title{The Fragmentation Equation with Size Diffusion: Well-Posedness and Long-Term Behavior} 
\thanks{Partially supported by Deutscher Akademischer Austauschdienst funding programme \textsl{Research Stays for University Academics and Scientists, 2021} (57552334)}
%
\author{Philippe Lauren\c{c}ot}
\address{Institut de Math\'ematiques de Toulouse, UMR~5219, Universit\'e de Toulouse, CNRS \\ F--31062 Toulouse Cedex 9, France}
\email{laurenco@math.univ-toulouse.fr}
\author{Christoph Walker}
\address{Leibniz Universit\"at Hannover\\ Institut f\"ur Angewandte Mathematik \\ Welfengarten 1 \\ D--30167 Hannover\\ Germany}
\email{walker@ifam.uni-hannover.de}
%
\date{\today}
%
\begin{abstract}
The dynamics of the fragmentation equation with size diffusion is investigated when the size ranges in $(0,\infty)$. The associated linear operator involves three terms and can be seen as a nonlocal perturbation of a Schr\"odinger operator. A Miyadera perturbation argument is used to prove that it is the generator of a positive, analytic semigroup on a weighted $L_1$-space. Moreover, if the overall fragmentation rate does not vanish at infinity, then there is a unique stationary solution with given mass. Assuming further that the overall fragmentation rate diverges to infinity for large sizes implies the immediate compactness of the semigroup and that it eventually stabilizes at an exponential rate to a one-dimensional projection carrying the information of the mass of the initial value. 
\end{abstract}
%
\maketitle
%
%
\pagestyle{myheadings}
\markboth{\sc{Ph.~Lauren\c{c}ot \& Ch.~Walker}}{\sc{Fragmentation-diffusion equation}}
%
%
\section{Introduction}\label{S.1}

The well-posedness of, along with the long-term behavior of solutions to, the fragmentation equation with size diffusion 
\begin{subequations}\label{FD.0}
\begin{align}
\partial_t \phi(t,x) - D \partial_x^2 \phi(t,x) & = - a(x) \phi(t,x) + \int_x^\infty a(y) b(x,y) \phi(t,y)\ \mathrm{d}y\,, \qquad (t,x)\in (0,\infty)^2\,, \label{FD.1} \\
\phi(t,0) & = 0\,, \qquad t>0\,, \label{FD.2} \\
\phi(0,x) & = f(x)\,, \qquad x\in (0,\infty)\,. \label{FD.3} 
\end{align}
\end{subequations}
is investigated by a semigroup approach. In \eqref{FD.0}, $\phi=\phi(t,x)\ge 0$ denotes the size distribution function of particles of size $x\in (0,\infty)$ at time $t>0$, while $a(x)\ge 0$ is the overall fragmentation rate of particles of size $x$, and $b(x,y)$ is the daughter distribution function which describes the distribution of fragments resulting from the breakup of a particle of size $y$. Besides undergoing fragmentation events, particles are also assumed to modify their size by diffusion at a constant diffusion rate $D>0$. Finally, nucleation is not taken into account in this model which leads to the homogeneous boundary condition \eqref{FD.2} at $x=0$. 

An interplay between diffusion and fragmentation as depicted by \eqref{FD.0} is met in the growth of ice crystals, see \cite{FJMOS2003, MFJLODS2004}. Indeed, on the one hand, ice crystals grow or shrink in a way which looks like diffusion and break apart due to internal stresses, the latter process being referred to as \textit{polygonization} or \textit{rotation recrystallization} in ice physics. The fragmentation equation with size diffusion \eqref{FD.0} is also derived in \cite{FHL1994} to describe the growth of microtubules. In the absence of diffusion (i.e. $D=0$), \Cref{FD.0} is the spontaneous fragmentation equation which has a long and rich history and has been extensively studied in the mathematical and physical literature since the pioneering works of \cite{Fil1961, McZi1987, ZiMc1985}, see \cite{BaAr2006, BLL2020a, BLL2020b, Ber2006, Ban2006}, and the references therein.

An important role is played in the dynamics by the total mass of the particles' distribution 
\begin{equation*}
M_1(\phi(t)) := \int_0^\infty x \phi(t,x)\ \mathrm{d}x\,, \qquad t\ge 0\,,
\end{equation*}
which is expected to be conserved throughout time evolution when there is no loss of matter during fragmentation events; that is, when $b$ satisfies
\begin{equation*}
	\int_0^y x b(x,y)\ \mathrm{d}x = y\,, \qquad y\in (0,\infty)\,. 
\end{equation*}
Thus, $X_1 := L_1((0,\infty),x\mathrm{d}x)$ is a natural functional framework for the study of the fragmentation operator. We further observe that the homogeneous Dirichlet boundary condition \eqref{FD.2} corresponds actually to a no-flux boundary condition for the Laplace operator in $X_1$, so that this space turns out to be also well-suited for diffusion. However, the analysis already performed on the fragmentation equation without diffusion reveals that a complete scale of weighted $L_1$-spaces is needed besides $X_1$. In this regard, we introduce the spaces 
$$
X_m := L_1((0,\infty),x^m\mathrm{d}x)\ \;\;\text{ and }\;\;  X_{1,m} := X_1\cap X_m
$$ 
for $m\in\mathbb{R}$. We denote the positive cone of $X_{1,m}$ by $X_{1,m}^+$. For $f\in X_m$ and $m\in\mathbb{R}$, we also define the moment $M_m(f)$ of order $m$ of $f$ by
\begin{equation*}
M_m(f) := \int_0^\infty x^m\ f(x)\ \mathrm{d}x\,,
\end{equation*}
so that $\|f\|_{X_m}=M_m(|f|)$. For definiteness, we equip $X_{1,m}$ with the norm $\|\cdot\|_{X_{1,m}} := \|\cdot\|_{X_1} + \|\cdot\|_{X_m}$ and note that $X_1 \doteq X_{1,1}$.

\medskip

Our strategy to study the well-posedness and the long-term behavior of \eqref{FD.0} is to write it as an abstract Cauchy problem in $X_{1,m}$ for $m\ge 1$ and show that the corresponding operator generates a semigroup with properties depending on $m$, $a$, and $b$. To this end, we assume throughout the paper that
\begin{equation}
a\in L_{\infty,loc}([0,\infty))\,, \qquad a\ge 0 \;\text{ a.e. in }\; (0,\infty)\,, \label{A.0}
\end{equation}
and that the daughter distribution function $b$ is a nonnegative measurable function on $(0,\infty)^2$ satisfying
\begin{equation}
	\int_0^y x b(x,y)\ \mathrm{d}x = y\,, \qquad y\in (0,\infty)\,. \label{B.0}
\end{equation}
Moreover, the diffusion rate $D$ is normalized to $D=1$.

For $m\ge 1$ we then define the (Schr\"odinger) operator $A_{a,m}$ on $X_{1,m}$ by
\begin{equation}
	\begin{split}
		\dom(A_{a,m}) & := \{ f\in X_{1,m}\ :\ f''\in X_{1,m}\,, \ af \in X_{1,m}\,, \ f(0)=0\}\,, \\
		A_{a,m}f & := f'' - a f\,, \qquad f\in \dom(A_{a,m})\,, 
	\end{split} \label{A.1}
\end{equation}	
as well as the nonlocal operator $B_{m}$ on $X_{1,m}$ by 
\begin{equation}
	\begin{split}
		\dom(B_m) & := \{ f\in X_{1,m}\ : \ af \in X_{1,m}\}\subset \dom(A_{a,m})\,, \\
		B_mf(x) & := \int_x^\infty a(y) b(x,y) f(y)\ \mathrm{d}y\,, \quad x\in (0,\infty)\,, \qquad f\in \dom(B_m)\,.
	\end{split}\label{B.1}
\end{equation}
Owing to \eqref{B.0} the operator $B_m$ turns out to be well-defined, see \Cref{L.B.1}. Setting
\begin{equation}
\mathbb{A}_m := A_{a,m}+B_m \;\text{ with }\; \dom(\mathbb{A}_m) := \dom(A_{a,m})\,, \label{AB.1}
\end{equation}
\Cref{FD.0} can be equivalently formulated as the Cauchy problem 
\begin{equation}\label{CP}
\frac{\mathrm{d}}{\mathrm{d}t} \phi = \mathbb{A}_m \phi\,, \quad t>0\,, \qquad \phi(0)=f\,,
\end{equation}
in $X_{1,m}$, and we shall investigate generation properties of the operator $\mathbb{A}_m$.

For concise statements we introduce the following notation: Given $\kappa\ge 1$ and $\omega\in\mathbb{R}$, we write $A\in \mathcal{G}(X_{1,m},\kappa,\omega)$ if the (unbounded) linear operator $A$ on $X_{1,m}$ is the generator of a strongly continuous semigroup $(e^{tA})_{t\ge 0}$ on $X_{1,m}$ and
$$
\|e^{tA}\|_{\mathcal{L}(X_{1,m})}\le \kappa e^{\omega t}\,,\qquad t\ge 0\,.
$$
We set
$$
\mathcal{G}(X_{1,m}):=\bigcup_{\kappa\ge 1\,,\, \omega\in\mathbb{R}}\mathcal{G}(X_{1,m},\kappa,\omega)\,.
$$
Moreover, we write $A\in \mathcal{G}_+(X_{1,m})$ if the semigroup $(e^{tA})_{t\ge 0}$ is positive  on the Banach lattice $X_{1,m}$. We denote the domain of the (unbounded) operator $A$ in $X_{1,m}$ by $\dom (A)$ and set
$$
D(A):=\big(\dom(A),\|\cdot\|_A\big)\,,
$$ 
where $\|f\|_A:=\|f\|_{X_{1,m}} + \|Af\|_{X_{1,m}}$ for $f\in \dom(A)$ is the graph norm. Finally, we write $A\in \mathcal{H}(X_{1,m})$ if $A\in \mathcal{G}(X_{1,m})$ and  the semigroup $(e^{tA})_{t\ge 0}$ is analytic.

\medskip

With this notation we may formulate the generation result in $X_{1,m}$ for $m\ge 1$:

\begin{theorem}\label{T.1}
Assume that $a$ and $b$ satisfy \eqref{A.0} and \eqref{B.0}.
\begin{itemize}
\item [(a)] There is an extension $\tilde{\mathbb{A}}_1\in \mathcal{G}_+(X_1,1,0)$ of $\mathbb{A}_1$. 

\item [(b)] Assume further that there is $\delta_2\in (0,1)$ such that
\begin{equation}
	(1-\delta_2) y^2 \ge \int_0^y x^2 b(x,y)\ \mathrm{d}x\,, \qquad y\in (0,\infty)\,. \label{B.10}
\end{equation}
If $m>1$, then $\mathbb{A}_m\in \mathcal{G}_+(X_{1,m})\cap \mathcal{H}(X_{1,m})$. In addition, for all $f\in X_{1,m}$,
\begin{equation}
M_1\left(e^{t\mathbb{A}_m}f \right) = M_1(f)\,, \qquad t\ge 0\,. \label{M.100}
\end{equation}

\item [(c)] Assume that $a$ satisfies 
\begin{equation}
	\lim_{x\to\infty} a(x) = \infty \label{C.0}
\end{equation}
and that $b$ satisfies \eqref{B.10}. Then, $(e^{t\mathbb{A}_m})_{t\ge 0}$ is immediately compact on $X_{1,m}$ for $m>1$.
\end{itemize}
\end{theorem}
 
Assumption~\eqref{B.10} is commonly encountered in the investigation of the fragmentation equation and somehow excludes the concentration of $b$ along the diagonal.

\medskip

At this stage, the extension $\tilde{\mathbb{A}}_1$ of $\mathbb{A}_1$ is not completely identified. In particular, we do not know whether or not $\tilde{\mathbb{A}}_1$ coincides with $\mathbb{A}_1$. Still, it is a question worth of investigation and we refer to \cite{BLL2020a} for a thorough discussion of this issue for the fragmentation equation without diffusion. Anyway, a positive answer is straightforward when $a\in L_\infty(0,\infty)$ and reported in the next result.

\begin{proposition}\label{P.10}
Assume that $a\in L_\infty(0,\infty)$ is non-negative and that $b$ satisfies \eqref{B.0}. 
\begin{itemize}
\item [(a)] For $m\ge 1$, $\mathbb{A}_m\in \mathcal{G}_+(X_{1,m})\cap \mathcal{H}(X_{1,m})$ and \eqref{M.100} is satisfied. In addition, $\mathbb{A}_1\in \mathcal{G}_+(X_1,1,0)$.
\item[(b)] The semigroup $(e^{t\mathbb{A}_m})_{t\ge 0}$ is not compact on $X_{1,m}$.
\end{itemize}
\end{proposition}

We immediately obtain the well-posedness of the Cauchy problem~\eqref{CP} in $X_{1,m}$ and, equivalently, of \eqref{FD.0} in a classical sense. Since we shall see that $\D(\mathbb{A}_m)\doteq \D(A_{a,m})$, we can formulate the result as follows:

\begin{corollary}\label{CO.1}
Assume that $a$ and $b$ satisfy \eqref{A.0} and \eqref{B.0}. Assume further that, either $m=1$ and $a\in L_\infty(0,\infty)$, or $m>1$ and $b$ satisfies \eqref{B.10}. Then, for any $f\in X_{1,m}$, there is a unique classical solution 
\begin{equation*}
\phi\in C([0,\infty),X_{1,m}) \cap C^1((0,\infty),X_{1,m}) \cap C((0,\infty),\D(A_{a,m}))
\end{equation*}
to \eqref{FD.0} which is given by $\phi(t) = e^{t\mathbb{A}_m}f$ for $t\ge 0$ and satisfies
\begin{equation}
M_1(\phi(t)) = M_1(f)\,, \qquad t\ge 0\,. \label{M.1}
\end{equation}
Moreover, if $f\ge 0$ a.e. in $(0,\infty)$, then $\phi(t)\ge 0$ for all $t>0$.
\end{corollary}

\begin{remark}\label{R.0}
Given $f\in X_{1,m}^+$, the corresponding solution $\phi$ to \eqref{FD.0} provided by \Cref{CO.1} satisfies the mass conservation \eqref{M.1}, a feature which is in particular due to the assumed boundedness \eqref{A.0} of $a$ on $(0,1)$. Indeed, even when $b$ satisfies \eqref{B.0}, infringement of mass conservation is known to occur when the overall fragmentation $a$ is unbounded for small sizes. In that case, the total mass is a decreasing function of time, a phenomenon usually referred to as shattering which is closely related to the honesty property of the associated semigroup, see \cite{Ban2004, BLL2020a, ChRe1990, Fil1961, Haa2003, Jeo2002, McZi1987}.
\end{remark}

Having settled the well-posedness of~\eqref{FD.0}, we next turn to qualitative properties of its dynamics.
As a guideline it was pointed out in \cite{FJMOS2003} that the interplay between diffusion and fragmentation results in the stabilization of solutions to~\eqref{FD.0}  to a stationary solution. This is in sharp contrast to the fragmentation equation without diffusion, since fragmentation is an irreversible process driving the particle distribution to zero. When diffusion is turned on,  a closed-form stationary solution to \eqref{FD.0} can be computed for the particular choice $a(x)=x$ and $b(x,y) = 2 y^{-1} \mathbf{1}_{(0,y)}(x)$, see \cite{FJMOS2003}. The existence of stationary solutions is also established in \cite{Lau2004} for an overall fragmentation rate $a$ obeying a power law ($a(x)=x^\gamma$, $\gamma\ge 0$) and for a specific class of daughter distribution  function $b$. Here we extend this existence result to a broader class of fragmentation coefficients $a$ and $b$, see \Cref{P.3}. In addition, when $a$ diverges to infinity as $x\to\infty$, we provide the exponential decay of the solution to \eqref{FD.0} to the steady state with the total mass of the initial value. 

\begin{theorem}\label{T.2}
Assume that $a$ and $b$ satisfy \eqref{A.0}, \eqref{B.0}, \eqref{B.10}, and \eqref{C.0}, and that $a>0$ and $b>0$.  There is a unique nonnegative
\begin{equation*}
\psi_1\in \bigcap_{r\ge 1} \dom(\mathbb{A}_{r})
\end{equation*} 
such that $M_1(\psi_1)=1$ and $\mathrm{ker}(\mathbb{A}_m) = \mathbb{R}\psi_1$ for every $m>1$. In particular, for $m>1$,  the spectral bound $s(\mathbb{A}_m)=0$  is a dominant eigenvalue of $\mathbb{A}_m$ and there are $N_m\ge 1$ and $\nu_m>0$ such that, for all $f\in X_{1,m}$,
\begin{equation*}
	\|e^{t  \mathbb{A}_m} f- M_1(f)\psi_1\|_{X_{1,m}}\le N_m e^{-\nu_m t} \| f\|_{X_{1,m}} \,,\quad t\ge 0 \,.
\end{equation*}
\end{theorem}
	
It is worth pointing out that the stationary solution $\psi_1$ decays faster than algebraically at infinity, a property which is perfectly consistent with the exponentially decaying tail experimentally observed in \cite{MFJLODS2004}. Also, combining \Cref{T.2} with \Cref{L.2.1} below implies that $\psi_1\in X_{r}$ for all $r>-1$.

\medskip

\Cref{T.2} provides a complete description of the long-term behavior of solutions to \eqref{FD.0} when $a$ diverges to infinity as $x\to \infty$. However, the unboundedness of $a$ at infinity is not a necessary condition for the existence of stationary solutions. In fact, when
	\begin{equation*}
		a(x) = 1\,, \qquad b(x,y) = \frac{2}{y} \mathbf{1}_{(0,y)}(x)\,, \qquad 0<x<y\,,
	\end{equation*}
we notice that
\Cref{FD.0}  has an explicit stationary solution $\psi_1(x) = x e^{-x}$, $x>0$.\footnote{We actually compute explicit stationary solutions to \eqref{FD.0} when 
$a(x)=x^\gamma$ and $b(x,y)=(\nu+2) x^\nu y^{-\nu-1}$ for $\gamma\ge 0$, $\nu\in (-2,0]$ in~\cite{LW21}.} 
This particular example is not peculiar and we actually obtain the existence of stationary solutions to \eqref{FD.0} as soon as there is a positive lower bound for $a$ as $x\to \infty$.

\begin{proposition}\label{P.3}
	Assume that $a$ and $b$ satisfy \eqref{A.0}, \eqref{B.0}, and \eqref{B.10}, and that $a>0$ and $b>0$.  Assume further that
	\begin{equation}
		\alpha := \frac{1}{2}\liminf_{x\to\infty} a(x) \in (0,\infty)\,. \label{St.1}
	\end{equation}
There is a unique nonnegative
	\begin{equation*}
		\psi_1\in \bigcap_{r\ge 1} \dom(\mathbb{A}_{r})
	\end{equation*} 
	such that $M_1(\psi_1)=1$ and $\mathrm{ker}(\mathbb{A}_m) = \mathbb{R}\psi_1$ for every $m>1$.
\end{proposition}

Let us mention here that there is no loss of generality in assuming the finiteness of $\liminf_{x\to\infty} a(x)$ in \eqref{St.1}. Indeed, if $\liminf_{x\to\infty} a(x)=\infty$, then $a$ satisfies \eqref{C.0}, a situation which is dealt with in \Cref{T.2}.

When $a$ only satisfies \eqref{St.1}, the associated semigroup $\left( e^{t  \mathbb{A}_m} \right)_{t\ge 0}$ need not be compact, see \Cref{P.10}. We thus take a different route to prove \Cref{P.3} an approximation procedure.
This approach does not allow us to retrieve information on the long-term behavior and it is likely that, either a more precise study of the operator $\mathbb{A}_1$, or a different approach (such as the one developed in \cite{MMP2005}) is required to fully identify the long-term behavior when $a$ only satisfies \eqref{St.1}.

\medskip

Let us end this introduction with a brief outline of the paper. Auxiliary results are gathered in the next section, which includes integrability properties of elements of $\dom(A_{0,1})$ on the one hand, and a weighted version of Kato's inequality on the other hand. In \Cref{S.2.1}, we recall some properties of the heat semigroup in the weighted $L_1$-space $X_{1,m}$  with $m\ge 1$, relying on the explicit representation formula which is available in that case. \Cref{S.2.2} is devoted to the Schr\"odinger operator $A_{a,m}$ and the associated absorption semigroup and is mostly a consequence of the thorough study performed in \cite{ArBa1993}. We then use a perturbation argument in \Cref{S.2.3} to study the full fragmentation-diffusion operator $\mathbb{A}_m=A_{a,m}+B_m$. On the one hand, for $m=1$, the existence of an extension $\tilde{\mathbb{A}}_1\in \mathcal{G}_+(X_1,1,0)$ of $\mathbb{A}_1$ is a consequence of \cite{Voi1987}. On the other hand, if $m>1$, then we can use a Miyadera perturbation technique to prove that $\mathbb{A}_m\in \mathcal{H}(X_{1,m})$. We recall that this approach has already proved successful for the fragmentation equation without size diffusion, see \cite{Ban2020}. The remainder of the paper is then devoted to the long-term dynamics. As a preliminary step, we first establish in \Cref{S.2.4} the immediate compactness of the semigroup in $X_{1,m}$ for $m>1$ when $a$ diverges to infinity as $x\to\infty$. We then construct in \Cref{S.2.5} a bounded convex subset of $X_{1,m}$ which is invariant with respect to the semigroup. This feature, along with the immediate compactness of the semigroup, implies the existence of at least one stationary solution for any given mass. After showing that the latter is unique, we perform a detailed study of the spectrum of $\mathbb{A}_m$ and end up with the announced convergence at a (yet non-explicit) exponential rate. Building upon the analysis performed in \Cref{S.2.5}, we turn to the proof of \Cref{P.3} in \Cref{sec.SM} which relies on an approximation procedure.
Specifically, introducing $a_n(x) := a(x) + x/n$ for $x>0$ and $n\ge 1$, we deduce from \Cref{T.2} that there is a unique  non-negative stationary solution $\psi_{1,n}$ to \eqref{FD.0} with $a_n$ instead of $a$. We then show that cluster points as $n\to\infty$ of this sequence are stationary solutions to \eqref{FD.0}.

\section{Auxiliary Results} \label{S.2}

According to the definition of $\dom(A_{0,1})$, an important role is played in the forthcoming analysis by functions $f\in X_1$ such that $f''\in X_1$. We collect useful properties of this class of functions in the next lemma and show, in particular, that the boundary condition \eqref{FD.3} is well-defined for such functions.

\begin{lemma}\label{L.2.1}
	Consider $f\in X_1$ such that $f''\in X_1$. Then $f\in BC([0,\infty))\cap C^1((0,\infty))$, $f'\in L_1(0,\infty)$, and, for $x> 0$,
	\begin{subequations}\label{P.1}
	\begin{equation}
		|f(x)| \le \|f''\|_{X_1}\,, \qquad x |f'(x)|\le \|f''\|_{X_1}\,, \qquad \|f'\|_{L_1(0,\infty)} \le \|f''\|_{X_1}\,. \label{P.1a}
	\end{equation}
	Moreover,
	\begin{equation}
		\lim_{x\to \infty} x f(x) = \lim_{x\to \infty} x f'(x) = 0\,. \label{P.1b}
	\end{equation}
	In fact, $f$ and $f'$ are given by
	\begin{equation}
		f(x) = - \int_x^\infty f'(y)\ \mathrm{d}y\,, \quad f'(x) = - \int_x^\infty f''(y)\ \mathrm{d}y\,, \qquad x\in (0,\infty)\,. \label{P.1c}
	\end{equation}
	\end{subequations}
	Also, $f\in X_m$ for any $m\in (-1,1)$ and, for all $\varepsilon>0$,
	\begin{equation}
		\|f\|_{X_m} \le \frac{\varepsilon^{m+1}}{m+1} \|f''\|_{X_1} + \varepsilon^{m-1} \|f\|_{X_1}\,. \label{P.2}
	\end{equation}
	Equivalently,
	\begin{equation}
		\|f\|_{X_m} \le \frac{2 (1-m)^{(m-1)/2}}{m+1} \|f''\|_{X_1}^{(1-m)/2} \|f\|_{X_1}^{(m+1)/2}\,,\qquad m\in (-1,1)\,. \label{P.2a}
	\end{equation}
\end{lemma}
	
\begin{proof}
	Introducing
	\begin{equation*}
		F(x) := \int_x^\infty (y-x) f''(y)\ \mathrm{d}y\,, \qquad x\in (0,\infty)\,, 
	\end{equation*}
	it follows from the integrability of $f''$ that 
	\begin{equation*}
		- \|f''\|_{X_1} \le - \int_x^\infty (y-x) |f''(y)| \mathrm{d}y \le F(x) \le \int_x^\infty (y-x) |f''(y)| \mathrm{d}y \le \|f''\|_{X_1}\,,
	\end{equation*}
	so that $F(x)$ is well-defined for $x\ge 0$. Moreover, $F$ belongs to $BC([0,\infty))\cap C^1((0,\infty))\cap W_{1,loc}^2((0,\infty))$ and satisfies 
	\begin{equation}
		F'(x) = - \int_x^\infty f''(y)\ \mathrm{d}y\,, \qquad F''(x) = f''(x)\,, \qquad x\in (0,\infty)\,, \label{P.3a}
	\end{equation}
	and
	\begin{equation}
		\lim_{x\to \infty} x F'(x) = \lim_{x\to \infty} F(x) = 0\,. \label{P.3b}
	\end{equation}
	In particular, we infer from \eqref{P.3a} that there is $(\alpha,\beta)\in\mathbb{R}^2$ such that $(f-F)(x) = \alpha + \beta x$ for $x>0$. Moreover, since $f\in X_1$, it follows from \eqref{P.3b} that
		\begin{equation*}
			\lim_{x\to\infty} \int_x^{x+1} |\alpha + \beta y|\ \mathrm{d}y \le \lim_{x\to\infty} \int_x^{x+1} (|f(y)|+|F(y)|)\ \mathrm{d}y = 0\,,
		\end{equation*}
		which readily gives $\alpha=\beta=0$ and $F=f$, thereby establishing \eqref{P.1}, except for the limiting behavior of $f$ at infinity. To this end, we observe that, since $f\in L_\infty(0,\infty)$ and $f'\in L_1(0,\infty)$, a formula for $f(x)^2$ reads
		\begin{equation*}
			f(x)^2 = - 2 \int_x^\infty f(y) f'(y)\ \mathrm{d}y\,, \qquad x\ge 0\,.
		\end{equation*}
		We then deduce from \eqref{P.1a} that
		\begin{equation*}
			f(x)^2 \le 2 \int_x^\infty y |f'(y)|  |f(y)| \ \frac{\mathrm{d}y}{y} \le \frac{2}{x^2} \|f''\|_{X_1} \int_x^\infty y |f(y)|\ \mathrm{d}y\,,
		\end{equation*}
		which implies that $x f(x)\to 0$ as $x\to\infty$ due to $f\in X_1$.
		
		Next, let $\varepsilon>\delta>0$. It follows from \eqref{P.1c} that
		\begin{align*}
			\int_\delta^\infty x^m |f(x)|\ \mathrm{d}x & \le \int_\delta^\varepsilon x^m |f(x)|\ \mathrm{d}x + \varepsilon^{m-1} \int_\varepsilon^\infty x\ |f(x)|\ \mathrm{d}x \\
			& \le \int_\delta^\varepsilon x^m \int_x^\infty |f'(y)|\ \mathrm{d}y \mathrm{d}x + \varepsilon^{m-1} \|f\|_{X_1} \\ 
			& \le \int_\delta^\varepsilon x^m \int_x^\varepsilon |f'(y)|\ \mathrm{d}y \mathrm{d}x + \frac{\varepsilon^{m+1}}{m+1} \int_\varepsilon^\infty |f'(y)|\ \mathrm{d}y + \varepsilon^{m-1} \|f\|_{X_1} \,.
		\end{align*}
		By  Fubini's theorem,
		\begin{equation*}
			\int_\delta^\varepsilon x^m \int_x^\varepsilon |f'(y)|\ \mathrm{d}y \mathrm{d}x = \frac{1}{m+1} \int_\delta^\varepsilon \left( y^{m+1} - \delta^{m+1} \right) |f'(y)|\ \mathrm{d}y \le \frac{\varepsilon^{m+1}}{m+1} \int_0^\varepsilon |f'(y)|\ \mathrm{d}y \,.
		\end{equation*}
Combining the above inequalities with \eqref{P.1a} gives
		\begin{equation*}
			\int_\delta^\infty x^m |f(x)|\ \mathrm{d}x \le \frac{\varepsilon^{m+1}}{m+1} \int_0^\infty |f'(y)|\ \mathrm{d}y + \varepsilon^{m-1} \|f\|_{X_1} \le \frac{\varepsilon^{m+1}}{m+1} \|f''\|_{X_1} + \varepsilon^{m-1} \|f\|_{X_1}\,.
		\end{equation*}
		Letting $\delta\to 0$ completes the proof of \eqref{P.2}. We next optimize \eqref{P.2} with respect to $\varepsilon\in (0,\infty)$ to derive \eqref{P.2a}.
	\end{proof}
	
	We next state for the sake of completeness the density in $X_{1,m}$ of smooth functions with compact support in $(0,\infty)$, which is actually a straightforward consequence of the density of $C_c^\infty((0,\infty))$ in $L_1(0,\infty)$. 
\begin{lemma}\label{L.2.2}
	The space $C_c^\infty((0,\infty))$ is dense in $X_{1,m}$ for $m\ge 1$. 
\end{lemma}
	
We finally recall a variant of the celebrated inequality of Kato \cite[Lemma~A]{Kat1972}.
	
\begin{lemma}\label{L.2.3}
	Let $\ell$ be a nonnegative function in $W_{\infty,loc}^1([0,\infty))$ and consider $f\in W_{1,loc}^2(0,\infty)$ such that $f'\in L_1((0,\infty),\ell'(x)\mathrm{d}x)$ and $f''\in L_1((0,\infty),\ell(x)\mathrm{d}x)$. Then
	\begin{equation}
		- \int_0^\infty \ell(x)\ \mathrm{sign}(f(x)) f''(x)\ \mathrm{d}x \ge \int_0^\infty \ell'(x)\ |f|'(x)\ \mathrm{d}x\,. \label{E.0}
	\end{equation}
\end{lemma}
	
	\begin{proof}
		For $\varepsilon\in (0,1)$, we define $\beta_\varepsilon\in W_{\infty,loc}^2(\mathbb{R})$ by $\beta_\varepsilon(0)=0$ and 
		\begin{equation*}
			\beta_\varepsilon'(r) := \mathrm{sign}(r)\,, \quad |r|>\varepsilon\,, \;\;\text{ and }\;\; \beta_\varepsilon'(r) := \frac{r}{\varepsilon}\,, \quad r\in [-\varepsilon,\varepsilon]\,.
		\end{equation*}
		Since $\beta_\varepsilon''\ge 0$ a.e. in $\mathbb{R}$, integration by parts gives
		\begin{align*}
			- \int_0^\infty \ell(x)\ \beta_\varepsilon'(f(x)) f''(x)\ \mathrm{d}x & = \int_0^\infty \left[ \ell(x)\ \beta_\varepsilon''(f(x)) |f'(x)|^2 + \ell'(x) \beta_\varepsilon'(f(x)) f'(x) \right]\ \mathrm{d}x \\
			& \ge \int_0^\infty \ell'(x) [\beta_\varepsilon(f)]'(x)\ \mathrm{d}x\,.
		\end{align*}
		Since $|\beta_\varepsilon(r)-r|\le \varepsilon$ for $r\in\mathbb{R}$ and $\beta_\varepsilon'$ converges pointwise to the $\mathrm{sign}$ function in $\mathbb{R}$ as $\varepsilon \to 0$, we may pass to the limit as $\varepsilon \to 0$ in the previous inequality to complete the proof.
	\end{proof}
	
	\section{The Heat Semigroup} \label{S.2.1}
	
	It is well-known that the solution to the heat equation with homogeneous Dirichlet boundary conditions at $x=0$,
	\begin{align*}
		\partial_t w - \partial_x^2 w & = 0\,, \qquad (t,x)\in (0,\infty)\times (0,\infty)\,, \\
		w(t,0) & = 0\,, \qquad t\in (0,\infty)\,, \\
		w(0,x) & = f(x)\,, \qquad x\in (0,\infty)\,,
	\end{align*}
	is given by the representation formula
	\begin{equation}
		w(t,x) = W(t)f(x) :=  \int_0^\infty [k(t,x-y)-k(t,x+y)] f(y)\ \mathrm{d}y\,, \qquad (t,x)\in (0,\infty)^2\,, \label{E.1}
	\end{equation}
	where
	\begin{equation}
		k(t,x):=\frac{1}{\sqrt{4\pi t}} e^{-\vert x\vert^2/4t}\,\qquad (t,x)\in (0,\infty)\times \mathbb{R}\,.\label{E.2}
	\end{equation}
	Moreover, $(W(t))_{t\ge 0}$ (with $W(0):=\mathrm{id}_{L_1(0,\infty)}$) is a positive analytic semigroup of contractions on $L_1(0,\infty)$ with generator $G$ given by
	\begin{equation*}
		\begin{split}
			\dom(G) & := \{ f\in L_1(0,\infty)\ : \ f''\in L_1(0,\infty)\;\text{ and }\; f(0) = 0 \}\,, \\
			G f & := f''\,, \qquad f\in \dom(G)\,.
		\end{split}
	\end{equation*}
	That is, $G\in\mathcal{G}_+(L_1(0,\infty),1,0)\cap \mathcal{H}(L_1(0,\infty))$ and $e^{tG}=W(t)$ for $t\ge 0$. We now consider this semigroup in the weighted space $X_{1,m}$. More precisely, for $m\ge 1$, we recall the definition \eqref{A.1} (with $a=0$) of the unbounded operator $A_{0,m}$ on $X_{1,m}$, given by
	\begin{equation}
		\begin{split}
			\dom(A_{0,m}) & = \{ f\in X_{1,m}\ : \ f''\in X_{1,m}\;\text{ and }\; f(0) = 0 \}\,, \\
			A_{0,m} f & = f''\,, \qquad f\in \dom(A_{0,m})\,,
		\end{split} \label{E.2.5}
	\end{equation}
and show that it is the generator of the heat semigroup in $X_{1,m}$.
	
	\begin{proposition}\label{P.2.4}
		Let $m\ge 1$. There is $\omega_m\ge 0$ such $A_{0,m}\in \mathcal{G}_+(X_{1,m},1,\omega_m)\cap \mathcal{H}(X_{1,m})$ with $(0,\infty)\subset \rho(A_{0,m})$. The semigroup $\left( e^{tA_{0,m}} \right)_{t\ge 0}$ is given by
		\begin{equation}
			e^{tA_{0,m}}f (x) = \int_0^\infty [k(t,x-y)-k(t,x+y)] f(y)\ \mathrm{d}y\,, \qquad (t,x)\in (0,\infty)\times (0,\infty)\,, \label{Rep}
		\end{equation}
		for $f\in X_{1,m}$, where $k$ is defined in \eqref{E.2}. Moreover, $e^{tA_{0,m}} = e^{tA_{0,1}}|_{X_{1,m}}$ for $t\ge 0$ and 
		\begin{equation*}
			\left( \lambda - A_{0,m} \right)^{-1} = \left( \lambda - A_{0,1} \right)^{-1}|_{X_{1,m}}\,, \qquad \lambda>0\,.
		\end{equation*}
	\end{proposition}
	
	Two steps are needed to show \Cref{P.2.4}. We first establish \Cref{P.2.4} for $m=1$ and $m\ge 3$, see \Cref{L.2.5} below. An interpolation argument then completes the proof for $m\in (1,3)$.
	
	\begin{lemma}\label{L.2.5}
		Let $m\in \{1\} \cup (3,\infty)$. Then $A_{0,m}\in \mathcal{G}_+(X_{1,m},1,\omega_m)\cap \mathcal{H}(X_{1,m})$ with  
		\begin{equation*}
			\omega_1 := 0 \;\text{ and }\; \omega_m := 4^{1/(m-1)} m (m-3)^{(m-3)/(m-1)}\,, \qquad m\ge 3\,.
		\end{equation*}
		Moreover, $e^{tA_{0,m}} = e^{tA_{0,1}}|_{X_{1,m}}$ for $t\ge 0$, $(0,\infty)\subset \rho(A_{0,m})$, and 
		$$
		(\lambda - A_{0,m})^{-1} = (\lambda-A_{0,1})^{-1}|_{X_{1,m}}\,,\qquad \lambda>0\,.
		$$
	\end{lemma}
	
	\begin{proof}
		We first note that $A_{0,m}$ is a closed operator on $X_{1,m}$ and that its domain is dense in $X_{1,m}$ due to \Cref{L.2.1} and \Cref{L.2.2}. We divide the remainder of the proof into several steps.
		
		\medskip
		
		\noindent\textbf{Step~1.} We first show the dissipativity of $A_{0,m}-\omega_m$ on $X_{1,m}$. To this end, let $\lambda>0$ and $f\in \dom(A_{0,m})$. By Kato's inequality \eqref{E.0} (with $\ell(x) = x^m$), \Cref{L.2.1}, and the boundary condition $f(0)=0$,
		\begin{align*}
			\|\lambda f - (A_{0,m}-\omega_m)f\|_{X_m} & \ge \int_0^\infty x^m\ \mathrm{sign}(f(x)) [(\lambda+\omega_m) f(x)-f''(x)]\ \mathrm{d}x \\
			& \ge (\lambda+\omega_m) \|f\|_{X_m} + m \int_0^\infty x^{m-1}\ |f|'(x)\ \mathrm{d}x \\
			& = (\lambda+\omega_m) \|f\|_{X_m} - m(m-1) \|f\|_{X_{m-2}}\,.
		\end{align*}
		In particular, when $m=1$,
		\begin{equation}
			\| \lambda f - A_{0,1}f\|_{X_1} \ge \lambda \|f\|_{X_1}\,, \label{E.7}
		\end{equation}
		so that $A_{0,1}$ is dissipative on $X_1$. We next handle the case $m\ge 3$. Then $m-2\in [1,m)$ and we infer from Young's inequality that, for $\varepsilon>0$,
		\begin{equation*}
			m(m-1) \|f\|_{X_{m-2}} \le m (m-3) \varepsilon \|f\|_{X_m} + 2m \varepsilon^{(3-m)/2} \|f\|_{X_1}\,.
		\end{equation*}
		Hence,
		\begin{equation*}
			(\lambda+\omega_m) \|f\|_{X_m} \le m (m-3) \varepsilon \|f\|_{X_m} + 2m \varepsilon^{(3-m)/2} \|f\|_{X_1} + \|\lambda f - (A_{0,m}-\omega_m)f\|_{X_m}\,.
		\end{equation*}
		Combining \eqref{E.7} with this inequality gives
		\begin{align*}
			(\lambda+\omega_m)  \|f\|_{X_{1,m}} & \le \|\lambda f - (A_{0,m}-\omega_m) f\|_{X_{1,m}}  \\
			& \qquad + m (m-3) \varepsilon \|f\|_{X_m} + 2m \varepsilon^{(3-m)/2} \|f\|_{X_1}\,.
		\end{align*}
		We now choose $\varepsilon = \varepsilon_m := (2/(m-3))^{2/(m-1)}$. Since
		\begin{equation*}
			\omega_m = m (m-3) \varepsilon_m = m \varepsilon_m^{(3-m)/2}\,,
		\end{equation*}
		we readily conclude that 
		\begin{equation*}
			\lambda  \|f\|_{X_{1,m}}  \le \|\lambda f - (A_{0,m}-\omega_m ) f\|_{X_{1,m}} \,,
		\end{equation*}
		so that $A_{0,m}-\omega_m I$ is a dissipative operator on $X_{1,m}$. 
		
		\medskip
	

\noindent\textbf{Step~2.} We next show that $\mathrm{rg}(\lambda - A_{0,m})=X_{1,m}$ for $\lambda>0$. Consider $g\in X_{1,m}$. According to \Cref{L.2.2}, there is a sequence $(g_n)_{n\ge 1}$ in $C_c^\infty((0,\infty))$ such that 
		\begin{equation}
			\lim_{n\to\infty} \|g_n-g\|_{X_{1,m}} = 0\,. \label{E.3}
		\end{equation}
Since $g_n\in L_1(0,\infty)$ and $(0,\infty)\subset \rho(G)$, there is a unique $f_n\in \dom(G)$ such that $\lambda f_n - G f_n = g_n$; that is, 
		\begin{equation}
			\lambda f_n - f_n'' = g_n \;\;\text{ in }\;\; (0,\infty)\,, \qquad f_n(0)=0\,. \label{E.4}
		\end{equation} 
Now, let $R>1$. We multiply \eqref{E.4} by $(x\wedge R)^m\ \mathrm{sign}(f_n(x))$  and integrate over $(0,\infty)$. Using Kato's inequality \eqref{E.0} (with $\ell(x)= (x\wedge R)^m$), we obtain
		\begin{align*}
			\|g_n\|_{X_m} & \ge \int_0^\infty (x\wedge R)^m\ \mathrm{sign}(f_n(x)) [\lambda f_n(x)-f_n''(x)]\ \mathrm{d}x \\
			& \ge \lambda \int_0^\infty (x\wedge R)^m\ |f_n(x)|\ \mathrm{d}x + m \int_0^R x^{m-1} |f_n|'(x)\ \mathrm{d}x \\
			& = \lambda \int_0^\infty (x\wedge R)^m\ |f_n(x)|\ \mathrm{d}x - m(m-1) \int_0^R x^{m-2}\ f_n(x)\ \mathrm{d}x \,.
		\end{align*}
In particular, for $m=1$ we get
		\begin{equation*}
			\|g_n\|_{X_1}  \ge \lambda \int_0^\infty (x\wedge R)\ |f_n(x)|\ \mathrm{d}x  \,,
		\end{equation*}
so that, letting $R\to \infty$ and using Fatou's lemma,		
		\begin{equation}\label{E.91}
			\lambda \|f_n\|_{X_1} \le \|g_n\|_{X_1}\,. 
		\end{equation}
	The same argument entails that
	\begin{equation}\label{E.91a}
			\lambda \|f_n-f_l\|_{X_1} \le \|g_n-g_l\|_{X_1}\,, \qquad n, l\ge 1\,.
		\end{equation}
If $m\ge 3$, then $m-2\in [1,m)$,  and we use  Young's inequality  to deduce that, for $\varepsilon>0$, 
		\begin{align*}
			\lambda \int_0^\infty (x\wedge R)^m\ |f_n(x)|\ \mathrm{d}x & \le \|g_n\|_{X_m} + m(m-3)\varepsilon \int_0^R x^m\ |f_n(x)|\ \mathrm{d}x\\
			& \qquad + 2m \varepsilon^{(3-m)/2} \int_0^R x\ |f_n(x)|\ \mathrm{d}x \\
			& \le \|g_n\|_{X_m} + m(m-3)\varepsilon \int_0^\infty (x\wedge R)^m\ |f_n(x)|\ \mathrm{d}x\\
			& \qquad + 2m \varepsilon^{(3-m)/2} \|f_n\|_{X_1}\,.
		\end{align*}
		Choosing $\varepsilon = \lambda/(2m(m-3))$, we combine \eqref{E.91} and the above inequality to conclude that
		\begin{equation*}
			\frac{\lambda}{2} \int_0^\infty (x\wedge R)^m\ |f_n(x)|\ \mathrm{d}x \le \|g_n\|_{X_m} + \frac{2m}{\lambda} \left( \frac{2m(m-3)}{\lambda} \right)^{(m-3)/2} \|g_n\|_{X_1}\,.
		\end{equation*}
		We now let $R\to\infty$ in this inequality and deduce from Fatou's lemma that $f_n\in X_m$ with 
		\begin{equation}\label{ui}
			\lambda \|f_n\|_{X_m} \le 2 \|g_n\|_{X_m} + \frac{4m}{\lambda} \left( \frac{2m(m-3)}{\lambda} \right)^{(m-3)/2} \|g_n\|_{X_1}\,.
		\end{equation}
The same argument entails that 
\begin{equation}\label{uii}
			\lambda \|f_n-f_l\|_{X_m} \le 2 \|g_n-g_l\|_{X_m} + \frac{4m}{\lambda} \left( \frac{2m(m-3)}{\lambda} \right)^{(m-3)/2} \|g_n-g_l\|_{X_1}
		\end{equation}
for $n\ge 1$ and $l\ge 1$. 
Therefore, for $m\in \{1\}\cup [3,\infty)$,  the estimates \eqref{E.91a} and \eqref{uii} along with \eqref{E.3} guarantee that $(f_n)_{n\ge 1}$ is a Cauchy sequence in $X_{1,m}$ and that there is $f\in X_{1,m}$ such that
		\begin{equation}
			\lim_{n\to\infty} \|f_n-f\|_{X_{1,m}} = 0\,. \label{E.5}
		\end{equation}
Since $f_n''=\lambda f_n-g_n$ for all $n\ge 1$ by \eqref{E.4}, it readily follows from \eqref{E.3} and \eqref{E.5} that $(f_n'')_{n\ge 1}$ converges to $\lambda f-g$ in $X_{1,m}$ and to $f''$ in the sense of distributions. 
 Therefore, $f''\in X_{1,m}$, $f''=\lambda f-g$, and $\|f_n''-f''\|_{X_{1,m}}\to 0$ as $n\to\infty$. 
Finally, by \eqref{P.1c},
		\begin{align*}
			|f(0)| & = |f(0)-f_n(0)| \le \int_0^\infty |(f'-f_n')(x)|\ \mathrm{d}x \\
			& \le \int_0^\infty \int_x^\infty |(f''-f_n'')(y)|\ \mathrm{d}y\mathrm{d}x = \int_0^\infty y\ |(f''-f_n'')(y)|\ \mathrm{d}y = \|f''-f_n''\|_{X_1}\,,
		\end{align*}
from which we deduce that $f(0)=0$. Consequently, $f\in \dom(A_{0,m})$ and  $\lambda f - A_{0,m} f = g$. Since $\lambda - A_{0,m}$ is one-to-one by \eqref{E.7}, we have thus shown for any $\lambda>0$ that 
\begin{equation}
		\mathrm{rg}(\lambda - A_{0,m})=X_{1,m} \;\text{ and }\; (\lambda-A_{0,m})^{-1}g = (\lambda - A_{0,1})^{-1}g \;\text{ for all }\; g\in X_{1,m}\,. \label{E.10}
\end{equation}

\medskip
		
\noindent\textbf{Step~3.} For $m\in \{1\}\cup [3,\infty)$, we infer from \textbf{Step~1} - \textbf{Step~2} and the Lumer-Phillips theorem \cite[Theorem~1.4.3]{Paz1983} that $A_{0,m}-\omega_m$ belongs to $\mathcal{G}(X_{1,m},1,0)$. Hence, $A_{0,m}\in \mathcal{G}(X_{1,m},1,\omega_m)$. Also, it readily follows from \eqref{E.10} that $(0,\infty)\subset \rho(A_{0,m})$ and $(\lambda - A_{0,m})^{-1} = (\lambda-A_{0,1})^{-1}|_{X_{1,m}}$. In particular, the latter, along with the exponential formula \cite[Theorem~1.8.3]{Paz1983}, entails that $e^{tA_{0,m}} = e^{tA_{0,1}}|_{X_{1,m}}$ for all $t\ge 0$. 
		
\medskip
		
\noindent\textbf{Step~4.} We next derive the representation formulation for $(e^{tA_{0,1}})_{t\ge 0}$. To this end, we first note that, for $t>0$, the operator $W(t)$ defined in \eqref{E.1} is a bounded operator on $X_1$. Indeed, since $|x-y|\le |x+y|=x+y$ for $(x,y)\in (0,\infty)^2$, 
\begin{equation}
		k(t,x-y) - k(t,x+y) =\frac{1}{\sqrt{4\pi t}} \left( e^{-|x-y|^2/4t} - e^{-|x+y|^2/4t} \right)\ge 0\,,\qquad (x,y)\in (0,\infty)^2\,.\label{E.6}
\end{equation}
By Fubini-Tonelli's theorem and \eqref{E.6},
\begin{align*}
	\|W(t)f\|_{X_1} & \le \int_0^\infty x \int_0^\infty [k(t,x-y) - k(t,x+y)] |f(y)|\ \mathrm{d}y\mathrm{d}x \\
	& = \frac{1}{\sqrt{\pi}} \int_0^\infty |f(y)| \left( \int_{-y/2\sqrt{t}}^\infty (y+2z\sqrt{t}) e^{-z^2}\ \mathrm{d}z - \int_{-\infty}^{-y/2\sqrt{t}} (-y-2z\sqrt{t}) e^{-z^2}\ \mathrm{d}z \right)\ \mathrm{d}y \\
	& = \frac{1}{\sqrt{\pi}} \int_0^\infty |f(y)| \left( \int_{-\infty}^\infty (y+2z\sqrt{t}) e^{-z^2}\ \mathrm{d}z \right)\ \mathrm{d}y = \|f\|_{X_1}\,.
\end{align*}
Thus, $W(t)$ is a contraction on $X_1$. Since $(\lambda - G)^{-1} = (\lambda  - A_{0,1})^{-1}$ on $L_1(0,\infty)\cap X_1$ for all $\lambda>0$, it follows from the exponential formula, see \cite[Theorem~1.8.3]{Paz1983}, that $e^{tG}=e^{tA_{0,1}}$ on $L_1(0,\infty)\cap X_1$ for all $t\ge 0$. Therefore, $W(t)=e^{tA_{0,1}}$ on $L_1(0,\infty)\cap X_1$ for all $t\ge 0$. Since $L_1(0,\infty)\cap X_1$ is dense in $X_1$ by \Cref{L.2.2} and $e^{tA_{0,1}}$ and $W(t)$ are both bounded operators on $X_1$, we conclude that $e^{tA_{0,1}}=W(t)$ on $X_1$ for all $t\ge 0$. Together with the outcome of \textbf{Step~3}, this identity proves \eqref{Rep}. In particular, $(e^{tA_{0,m}})_{t\ge 0}$ is a positive semigroup according to \eqref{E.6}. 
		
\medskip
		
\noindent\textbf{Step~5.} We are left with showing the analyticity of $(e^{tA_{0,m}})_{t\ge 0}$ on $X_{1,m}$. To this end, let $f\in X_{1,m}$. Clearly, 
\begin{equation*}
	t \longmapsto \int_0^\infty [k(t,x-y)-k(t,x+y)] f(y)\ \mathrm{d}y
\end{equation*}
is differentiable on $(0,\infty)$ and, since
\begin{equation*}
	2t \partial_t k(t,x) = \left( - 1 + \frac{|x|^2}{2t} \right) k(t,x)\,, \qquad (t,x)\in (0,\infty)\times\mathbb{R}\,,
\end{equation*}
we derive for $(t,x)\in (0,\infty)^2$ that
\begin{align*}
	2t\frac{\mathrm{d}}{\mathrm{d}t} e^{tA_{0,m}}f(x) & = - e^{tA_{0,m}}f(x) \\
	& \qquad + 2 \int_0^\infty \left( \frac{|x-y|^2}{4t} k(t,x-y) - \frac{|x+y|^2}{4t} k(t,x+y) \right) f(y)\ \mathrm{d}y \\
	& = - 3 e^{tA_{0,m}}f(x) \\
	& \qquad + 2 \int_0^\infty \left[ \left( 1 + \frac{|x-y|^2}{4t} \right) k(t,x-y) - \left( 1+ \frac{|x+y|^2}{4t} \right) k(t,x+y) \right] f(y)\ \mathrm{d}y\,.
\end{align*}
Let $t>0$. Since $z\mapsto (1+z) e^{-z}$ is non-increasing on $(0,\infty)$ and $|x-y|\le x+y$ for $(x,y)\in (0,\infty)^2$, we see that
\begin{equation}
	\left( 1 + \frac{|x-y|^2}{4t} \right) k(t,x-y) - \left( 1+ \frac{|x+y|^2}{4t} \right) k(t,x+y) \ge 0\,, \qquad (x,y)\in (0,\infty)^2\,, \label{E.6.5.1}
\end{equation}
and infer from Fubini-Tonelli's theorem that
\begin{align*}
	\int_0^\infty x^m & \left| \int_0^\infty \left[ \left( 1 + \frac{|x-y|^2}{4t} \right) k(t,x-y) - \left( 1+ \frac{|x+y|^2}{4t} \right) k(t,x+y) \right] f(y)\ \mathrm{d}y \right|\ \mathrm{d}x \\ 
	&  \le \int_0^\infty x^m \int_0^\infty \left[ \left( 1 + \frac{|x-y|^2}{4t} \right) k(t,x-y) - \left( 1+ \frac{|x+y|^2}{4t} \right) k(t,x+y) \right] |f(y)|\ \mathrm{d}y\mathrm{d}x \\
	&  = \frac{1}{\sqrt{\pi}} \int_0^\infty |f(y)| \left[ \int_{-y/2\sqrt{t}}^\infty \left( y+2z\sqrt{t} \right)^m (1+z^2) e^{-z^2}\ \mathrm{d}z \right. \\ 
	& \hspace{5cm} \left. - \int_{-\infty}^{-y/2\sqrt{t}} \left( -y-2z\sqrt{t} \right)^m (1+z^2) e^{-z^2}\ \mathrm{d}z \right]\ \mathrm{d}y \\
	&  = \frac{1}{\sqrt{\pi}} \int_0^\infty |f(y)| \int_{-\infty}^\infty \left| y+2z\sqrt{t} \right|^{m-1} \left( y+2z\sqrt{t} \right) (1+z^2) e^{-z^2}\ \mathrm{d}z\mathrm{d}y \,.
\end{align*}
Consequently,
\begin{align*}
	2t \left\| \frac{\mathrm{d}}{\mathrm{d}t} e^{tA_{0,m}}f \right\|_{X_{1,m}} & \le 3 \left\| e^{tA_{0,m}}f \right\|_{X_{1,m}} \\
	&  \qquad + \frac{2}{\sqrt{\pi}} \int_0^\infty |f(y)| \int_{-\infty}^\infty  \left( y+2z\sqrt{t} \right) (1+z^2) e^{-z^2}\ \mathrm{d}z\mathrm{d}y \\
	& \qquad + \frac{2}{\sqrt{\pi}} \int_0^\infty |f(y)|   \int_{-\infty}^\infty \left| y+2z\sqrt{t} \right|^{m-1} \left( y+2z\sqrt{t} \right) (1+z^2) e^{-z^2}\ \mathrm{d}z \, \mathrm{d}y
\end{align*}
so that
\begin{equation*}
	\limsup_{t\to 0} t \left\| \frac{\mathrm{d}}{\mathrm{d}t} e^{tA_{0,m}}f \right\|_{X_{1,m}} \le 3 \|f\|_{X_{1,m}}
\end{equation*}
by Lebesgue's convergence theorem. It is well-known that this property implies the analyticity of $(e^{tA_{0,m}})_{t\ge 0}$ on $X_{1,m}$. Thus, the proof is complete.
\end{proof}

The proof of \Cref{P.2.4} is now a consequence of the previous lemma and an interpolation argument as shown next.
	
\begin{proof}[Proof of \Cref{P.2.4}]
We only have to consider the case $m\in (1,3)$. Since $A_{0,1}\in\mathcal{H}( X_{1,1})$ and $A_{0,3}\in\mathcal{H}(X_{1,3})$ according to \Cref{L.2.5}, Hille's characterization implies that there are $\lambda_0>0$ and $\kappa\ge 1$ such that
\begin{equation*}
	\|(\lambda-A_{0,1})^{-1}\|_{\mathcal{L}(X_{1,1})} +\|(\lambda-A_{0,3})^{-1}\|_{\mathcal{L}(X_{1,3})}\le \frac{\kappa}{\vert\lambda-\lambda_0\vert}\,,\quad \mathsf{Re}\, \lambda> \lambda_0\,.
\end{equation*}
Since $(\lambda-A_{0,3})^{-1}=(\lambda-A_{0,1})^{-1}|_{X_{1,3}}$, it readily follows by interpolation that
\begin{equation}\label{M1}
	\|(\lambda-A_{0,1})^{-1}|_{X_{1,m}}\|_{\mathcal{L}(X_{1,m})} \le  \frac{\kappa}{\vert\lambda-\lambda_0\vert}\,,\quad \mathsf{Re}\, \lambda> \lambda_0\,.
\end{equation}
Obviously, $A_{0,m}$ is closed and densely defined in $X_{1,m}$.		Let $g\in X_{1,m}$ be arbitrary and $\mathsf{Re}\, \lambda> \lambda_0$. There is a unique $f\in \dom(A_{0,1})$ such that $(\lambda-A_{0,1})f=g$. By \eqref{M1}, $f=(\lambda-A_{0,1})^{-1}|_{X_{1,m}}g\in X_{1,m}$ and $f''=A_{0,m}f=\lambda f-g\in X_{1,m}$. From this identity, we deduce $f\in \dom(A_{0,m})$ with $A_{0,m} f=A_{0,1}f$ and $(\lambda-A_{0,m})f=g$. Since $f$ is unique, we conclude that $\lambda\in\rho(A_{0,m})$ and 
\begin{equation*}
	(\lambda-A_{0,m})^{-1}g=f=(\lambda-A_{0,1})^{-1}|_{X_{1,m}}g\,;
\end{equation*}
that is, $(\lambda-A_{0,m})^{-1}=(\lambda-A_{0,1})^{-1}|_{X_{1,m}}$. Invoking \eqref{M1} we derive that
\begin{equation*}
	\|(\lambda-A_{0,m})^{-1}\|_{\mathcal{L}(X_{1,m})} \le  \frac{\kappa}{\vert\lambda-\lambda_0\vert}\,,\quad \mathsf{Re}\, \lambda> \lambda_0\,,
\end{equation*}
and thus conclude that $A_{0,m}\in\mathcal{H}(X_{1,m})$ with $e^{tA_{0,m}} = e^{tA_{0,1}}|_{X_{1,m}}$ for $t\ge 0$. The fact that $A_{0,m} \in \mathcal{G}(X_{1,m},1,\omega_m)$ follows again by interpolation using \Cref{L.2.5}.
\end{proof}
	
\section{The Absorption Semigroup} \label{S.2.2}
	
Let $a\in L_{\infty,loc}([0,\infty))$ be a non-negative function and $m\ge 1$. We recall the definition \eqref{A.1} of the Schr\"odinger operator $A_{a,m}$ on $X_{1,m}$ given by
\begin{equation*}
	\begin{split}
		\dom(A_{a,m}) & = \{ f\in X_{1,m}\ :\ f\in \dom(A_{0,m})\,, \ af \in X_{1,m}\}\,, \\
		A_{a,m}f & = f'' - a f\,, \qquad f\in \dom(A_{a,m})\,. 
	\end{split} 
\end{equation*}
	
The main result of this section is that $A_{a,m}\in \mathcal{G}_+(X_{1,m},1,\omega_m)\cap \mathcal{H}(X_{1,m})$, the parameter $\omega_m$ being defined in \Cref{P.2.4}. The proof relies on \cite{ArBa1993}.
	
\begin{proposition}\label{P.A.1}
Assume that  $a$ satisfies \eqref{A.0} and let $m\ge 1$. Then $A_{a,m}\in \mathcal{G}_+(X_{1,m},1,\omega_m)\cap \mathcal{H}(X_{1,m})$. Moreover, $e^{tA_{a,m}} = e^{tA_{a,1}}|_{X_{1,m}}$ for $t\ge 0$.
\end{proposition}
	
\begin{proof}
Let us recall that $X_{1,m}$ is a Banach lattice with order-continuous norm and introduce $Y := L_1((0,\infty),(x+x^m)a(x)\mathrm{d}x)$. We first observe that $\left( \dom(A_{0,m})\cap Y \right)^\perp = \{0\}$, where the disjoint complement $F^\perp$ of a subset $F$ of $X_{1,m}$ is given by
\begin{equation*}
		F^\perp := \{ g \in X_{1,m}\ :\ \min\{|f|,|g|\} = 0 \;\text{ for all }\; f\in F\}\,.
\end{equation*} 
Indeed,  since $C_c^\infty((0,\infty))$ is a subset of $\dom(A_{0,m})\cap Y$, we readily deduce that $g\equiv 0$ for $g\in \left( \dom(A_{0,m})\cap Y \right)^\perp$. Consequently, we are in a position to apply \cite[Proposition~4.3]{ArBa1993} and conclude that there is an extension $\hat{A}_{a,m}\in \mathcal{G}_+(X_{1,m})$  of $A_{a,m}$ with domain
\begin{equation*}
		\dom(\hat{A}_{a,m}) := \left\{
		\begin{array}{cc}
			f\in X_{1,m}\ :\ & \text{there exist $(f_n)_{n\ge 1}$ in $\dom(A_{0,m})$ and $g\in X_{1,m}$ such that}\\
			&  \\
			& \displaystyle{\lim_{n\to\infty} \left( \|f_n-f\|_{X_{1,m}} + \|A_{0,m}f_n - (a\wedge n) f_n + g \|_{X_{1,m}} \right) = 0}
		\end{array}
		\right\}\,.
\end{equation*}
It first follow from \Cref{ACDC} below that
$\dom(\hat{A}_{a,m})=\dom(A_{a,m})$ and therefore $\hat{A}_{a,m}=A_{a,m}$. Moreover,
 $0\le e^{t A_{a,m}}=e^{t \hat{A}_{a,m}}\le e^{t A_{0,m}}$ for $t\ge 0$ by \cite[p.432]{ArBa1993}. Since $A_{0,m}\in\mathcal{G}_+(X_{1,m},1,\omega_m)$ due to \Cref{P.2.4}, this ordering property, along with \cite[Remark 2.68]{BaAr2006}, implies
$$
\| e^{t \hat{A}_{a,m}}\|_{\mathcal{L}(X_{1,m})}\le \| e^{t A_{0,m}}\|_{\mathcal{L}(X_{1,m})}\le e^{\omega_m t}\,,\qquad t\ge 0\,.
$$ 
Hence $\hat{A}_{a,m}\in \mathcal{G}_+(X_{1,m},1,\omega_m)$.
 Finally, recalling that
$$
(-\omega_m+A_{0,m})\in\mathcal{G}_+(X_{1,m},1,0)\cap \mathcal{H}(X_{1,m})
$$  
by \Cref{P.2.4}, we infer from \cite[Theorem~6.1]{ArBa1993} that $A_{a,m}\in \mathcal{H}(X_{1,m})$.
\end{proof}
	
It  remains to check that $\dom(\hat{A}_{a,m})=\dom(A_{a,m})$.  This property actually follows from the monotonicity of the Laplace operator $A_{0,m}$ and the multiplication  $f\mapsto af$.
	
\begin{lemma}\label{ACDC}
Assume that  $a$ satisfies \eqref{A.0} and let $m\ge 1$. Then		 $\dom(\hat{A}_{a,m})=\dom(A_{a,m})$.
\end{lemma}

\begin{proof}
Pick $f\in \dom(\hat{A}_{a,m})$. Then there are a sequence $(f_n)_{n\ge 1}$ in $\dom(A_{0,m})$ and $g\in X_{1,m}$ such that
\begin{equation}
	\lim_{n\to\infty} \left( \|f_n-f\|_{X_{1,m}} + \|g_n - g \|_{X_{1,m}} \right) = 0\,, \label{A.2}
\end{equation}
with $g_n := - A_{0,m}f_n + (a\wedge n) f_n$ for $n\ge 1$. In particular,
\begin{equation*}
	\kappa := \sup_{n\ge 1}\left\{ \|f_n\|_{X_{1,m}} + \|g_n\|_{X_{1,m}} \right\} < \infty\,. 
\end{equation*}
		
\medskip
				
\noindent\textbf{Step~1.} Let us first prove that $af\in X_{1,m}$. Indeed, we infer from \Cref{L.2.3} that, for $n\ge 1$, 
\begin{align*}
	\|g_n\|_{X_m} & \ge \int_0^\infty x^m\, \mathrm{sign}(f_n(x))\, g_n(x)\ \mathrm{d}x \\
	& \ge m \int_0^\infty x^{m-1}\ |f_n|'(x)\ \mathrm{d}x + \int_0^\infty x^m\ (a(x)\wedge n) |f_n(x)|\ \mathrm{d}x \\
	& = - m(m-1) \|f_n\|_{X_{m-2}} + \int_0^\infty x^m\ (a(x)\wedge n) |f_n(x)|\ \mathrm{d}x\,.
\end{align*}
In particular,  we derive
\begin{equation}\label{lp}
\int_0^\infty x\ (a(x)\wedge n) |f_n(x)|\ \mathrm{d}x\le \|g_n\|_{X_1}\le \kappa\,.
\end{equation}
Next, if $m\ge 3$, then it follows from Young's inequality that 
\begin{align*}
	\int_0^\infty  x^m\ (a(x)\wedge n) |f_n(x)|\ \mathrm{d}x & \le m(m-1) \left[ \frac{m-3}{m-1} \|f_n\|_{X_m} + \frac{2}{m-1} \|f_n\|_{X_1} \right] + \|g_n\|_{X_m} \\
	& \le [m(m-1)+1] \kappa\,.
\end{align*}
Likewise, if $m\in (1,3)$, then \Cref{L.2.1} implies that
\begin{align*}
	\int_0^\infty x^m\ (a(x)\wedge n) |f_n(x)|\ \mathrm{d}x & \le m(m-1) \left[ \frac{1}{m-1} \|f_n''\|_{X_1} + \|f_n\|_{X_1} \right] + \|g_n\|_{X_m} \\
	& \le m \|(a\wedge n) f_n -g_n\|_{X_1} + [m(m-1)+1]\kappa \\
	& \le m \sup_{l\ge 1} \{\|(a\wedge l) f_l\|_{X_1}\} + (m^2+1) \kappa\\
	& \le(m^2+m+1) \kappa \,,
\end{align*}
where the last inequality is due to  \eqref{lp}. Thus, in all cases for $m\ge 1$, we have shown that
\begin{equation*}
	\int_0^\infty (x+x^m)\ (a(x)\wedge n) |f_n(x)|\ \mathrm{d}x \le(m^2+m+2) \kappa\,.
\end{equation*}
 Fixing $N\ge 1$, we deduce from the previous estimate that, for all $n\ge N$,
\begin{equation*}
	\int_0^\infty (x+x^m)\ (a(x)\wedge N) |f_n(x)|\ \mathrm{d}x \le \int_0^\infty x^m\ (a(x)\wedge n) |f_n(x)|\ \mathrm{d}x \le (m^2+m+2) \kappa\,.
\end{equation*}
We then let $n\to\infty$ and infer from \eqref{A.2} that
\begin{equation*}
	\int_0^\infty (x+x^m)\ (a(x)\wedge N) |f(x)|\ \mathrm{d}x \le (m^2+m+2) \kappa\,.
\end{equation*}
Using Fatou's lemma to let $N\to\infty$, we conclude that $af\in X_{1,m}$ with
\begin{equation}
	\|a f\|_{X_{1,m}} \le  (m^2+m+2) \kappa\,. \label{A.6a}
\end{equation}   

\medskip
		
\noindent\textbf{Step~2.} We next show that $((a\wedge n) f_n)_{n\ge 1}$ converges to $af$ in  $X_{1,m}$. Let $\chi\in C^\infty((0,\infty))$ be such that $\chi(x)=1$ for $x> 2$, $\chi(x)=0$ for $x\in (0,1)$, and $\chi(x)\in [0,1]$ for $x\in [1,2]$. Introducing $\chi_R(x) := \chi(x/R)$ for $x\in (0,\infty)$ and $R>1$, we deduce from \Cref{L.2.3} (with $\ell(x) = x^m \chi_R(x)$) that
\begin{align*}
	\int_0^\infty x^m\ (a(x)\wedge n) \chi_R(x) |f_n(x)|\ \mathrm{d}x & + \int_0^\infty [x^m \chi_R'(x) + m x^{m-1} \chi_R(x)] |f_n|'(x)\ \mathrm{d}x \\
	& \hspace{1cm} \le \int_0^\infty x\chi_R(x) \mathrm{sign}(f_n(x)) g_n(x)\ \mathrm{d}x \\
	& \hspace{1cm} \le \int_R^\infty x\ |g_n(x)|\ \mathrm{d}x\,.
\end{align*}
Since 
\begin{align*}
	& \int_0^\infty [x^m \chi_R'(x) + m x^{m-1} \chi_R(x)] |f_n|'(x)\ \mathrm{d}x \\
	& \hspace{2cm} = - \int_0^\infty [x^m \chi_R''(x) + 2m x^{m-1} \chi_R'(x) + m(m-1) x^{m-2} \chi_R(x)] |f_n(x)|\ \mathrm{d}x \\
	& \hspace{2cm} = - \int_0^\infty x^m\ |f_n(x)| \left[ \frac{1}{R^2} \chi''\left( \frac{x}{R} \right) + \frac{2m}{Rx} \chi'\left( \frac{x}{R} \right) + \frac{m(m-1)}{x^2} \chi\left( \frac{x}{R} \right) \right]\ \mathrm{d}x 
\end{align*}
and
\begin{equation*}
	\left| \chi''(y) + \frac{2m}{y}\chi'(y) + \frac{m(m-1)}{y^2} \chi(y)\right| \le 3m^2 \|\chi''\|_{L_\infty(0,\infty)}\,, \qquad y\in (0,\infty)\,,
\end{equation*}
we further obtain
\begin{align*}
	\int_{2R}^\infty x^m\ (a(x)\wedge n) |f_n(x)|\ \mathrm{d}x & \le \int_0^\infty x^m\ (a(x)\wedge n) \chi_R(x) |f_n(x)|\ \mathrm{d}x \\
	& \le \sup_{l\ge 1} \left\{ \int_R^\infty x^m\ |g_l(x)|\ \mathrm{d}x \right\} + \frac{3m^2}{R^2} \|\chi''\|_{L_\infty(0,\infty)} \|f_n\|_{X_m} \\
	& \le \sup_{l\ge 1} \left\{ \int_R^\infty x^m\ |g_l(x)|\ \mathrm{d}x \right\} + \frac{3m^2\kappa}{R^2} \|\chi''\|_{L_\infty(0,\infty)}
\end{align*}
for $n\ge 1$. Since
\begin{equation*}
	\lim_{R\to\infty} \sup_{l\ge 1}\left\{ \int_R^\infty (x+x^m)\ |g_l(x)|\ \mathrm{d}x \right\} = 0
\end{equation*}
by \eqref{A.2}, we conclude
\begin{equation}\label{A.4a}
	\lim_{R\to\infty} \sup_{n\ge 1} \left\{ \int_{2R}^\infty (x+x^m)\ (a(x)\wedge n) |f_n(x)|\ \mathrm{d}x \right\}= 0\,. 
\end{equation}
Now, let $R>1$. Since $a\in L_\infty(0,2R)$ by \eqref{A.0}, there is $n_R\ge 1$ such that $a(x)\wedge n=a(x)$ for $x\in (0,2R)$ and $n\ge n_R$. Consequently, for $n\ge n_R$,
\begin{align*}
	\|(a\wedge n) f_n - af\|_{X_{1,m}} & \le \int_0^{2R} (x+x^m)\ a(x) |(f_n-f)(x)|\ \mathrm{d}x + \int_{2R}^\infty (x+x^m)\ (a(x)\wedge n) |f_n(x)|\ \mathrm{d}x \\
	& \hspace{2cm} + \int_{2R}^\infty (x+x^m)\ a(x) |f(x)|\ \mathrm{d}x \\
	& \le \|a\|_{L_\infty(0,2R)} \|f_n-f\|_{X_{1,m}} + \sup_{l\ge 1} \left\{ \int_{2R}^\infty (x+x^m)\ (a(x)\wedge l) |f_l(x)|\ \mathrm{d}x \right\} \\
	& \hspace{2cm} + \int_{2R}^\infty (x+x^m)\ a(x) |f(x)|\ \mathrm{d}x \,.
\end{align*}
We then pass to the limit as $n\to\infty$ and infer from \eqref{A.2} that
\begin{align*}
	\limsup_{n\to\infty} \|(a\wedge n) f_n - af\|_{X_{1,m}} &\le \sup_{l\ge 1} \left\{ \int_{2R}^\infty (x+x^m)\ (a(x)\wedge l) |f_l(x)|\ \mathrm{d}x \right\}\\
	&\qquad + \int_{2R}^\infty (x+x^m)\ a(x) |f(x)|\ \mathrm{d}x\,.
\end{align*}
We finally let $R\to\infty$ with the help of \eqref{A.6a} and \eqref{A.4a} and end up with
\begin{equation}
	\lim_{n\to\infty} \|(a\wedge n) f_n - af\|_{X_{1,m}} = 0\,. \label{A.7}
\end{equation}

\medskip
		
\noindent\textbf{Step~3.} We finally show that $f\in \dom( A_{0,m})$. Indeed, it readily follows from \eqref{A.2} that $(f_n'')_{n\ge 1}$ converges to $f''$ in the sense of distributions, while \eqref{A.2}, \eqref{A.6a}, and \eqref{A.7}  guarantee that $(f_n'')_{n\ge 1}=((a\wedge n) f_n - g_n)_{n\ge 1}$ converges to $af -g$ in $X_{1,m}$. Therefore, $f''$ belongs to $X_{1,m}$ with $f''=af-g$ and $(f_n'')_{n\ge 1}$ converges to $f''$ in $X_{1,m}$. Since $f_n(0)=0$ for $n\ge 1$, this convergence along with \Cref{L.2.1} ensures that $f(0)=0$ and we have proved that $f\in \dom(A_{0,m})$.
\end{proof}
	
For further use, we show that the graph norm of $A_{a,m}$ in $X_{1,m}$ controls independently the diffusive and absorption terms in $X_{1,m}$.
	
\begin{lemma}\label{L.A.2} 
	Assume  \eqref{A.0}. For $f\in \dom(A_{a,1})$,
	\begin{equation}
		\frac{1}{3} \left( \|A_{0,1}f\|_{X_1} + \|af\|_{X_1} \right) \le \|A_{a,1}f\|_{X_1}\,. \label{A.8a}
	\end{equation}
	Let $m>1$. For $f\in \dom(A_{a,m})$, 
	\begin{equation}
		\frac{1}{4(m+1)} \left( \|A_{0,m}f\|_{X_{1,m}} + \|af\|_{X_{1,m}} \right) - m \|f\|_{X_{1,m}} \le \|A_{a,m}f\|_{X_{1,m}}\,. \label{A.8b}
	\end{equation}
\end{lemma}
	
\begin{proof}
Let $f\in \dom(A_{a,1})$ and set $g:= -A_{a,1}f = -f''+af$. It follows from \Cref{L.2.3} (with $\ell(x)=x$) that
\begin{equation*}
	\|g\|_{X_1} \ge \int_0^\infty x\ \mathrm{sign}(f(x)) g(x)\ \mathrm{d}x \ge \int_0^\infty |f|'(x)\ \mathrm{d}x + \|a f\|_{X_1} = \|a f\|_{X_1}\,.
\end{equation*}
Consequently, 
\begin{equation}
	\|a f\|_{X_1} \le \|A_{a,1}f\|_{X_1} \;\text{ and }\; \|A_{0,1}f\|_{X_1} = \|A_{a,1}f + af\|_{X_1}\le 2 \|A_{a,1}f\|_{X_1}\,, \label{A.9}
\end{equation}
from which we deduce \eqref{A.8a}. 
		
Next, let $m>1$ and consider $f\in \dom(A_{a,m})$. We set $g:= -A_{a,m}f = -f''+af$ and infer from \Cref{L.2.3} (with $\ell(x)=x^m$) that
\begin{align*}
	\|g\|_{X_m} & \ge \int_0^\infty x^m\ \mathrm{sign}(f(x)) g(x)\ \mathrm{d}x \ge m \int_0^\infty x^{m-1} |f|'(x)\ \mathrm{d}x + \|a f\|_{X_m} \\
	& = \|a f\|_{X_m} - m(m-1) \|f\|_{X_{m-2}}\,.
\end{align*}
\begin{subequations}\label{A.10}
	Either $m\ge 3$ and it follows from Young's inequality and the above inequality that
	\begin{align}
		\|a f\|_{X_m} & \le \|g\|_{X_m} + m(m-1) \left( \frac{m-3}{m-1} \|f\|_{X_m} + \frac{2}{m-1} \|f\|_{X_1} \right) \nonumber \\
		& \le \|g\|_{X_m} + m(m-3) \|f\|_{X_m} + 2m \|f\|_{X_1} \nonumber \\
		& \le \|g\|_{X_m} + m^2 \|f\|_{X_{1,m}} \,. \label{A.10a}
	\end{align}
	Or $m\in (1,3)$ and we infer from \eqref{P.2} and \eqref{A.9} that
	\begin{align}
		\|a f\|_{X_m} & \le \|g\|_{X_m} + m(m-1) \left( \frac{1}{m-1}  \|f''\|_{X_1} + \|f\|_{X_1} \right) \nonumber \\
		& \le \|g\|_{X_m} + m \|A_{0,1}f\|_{X_1} + m(m-1) \|f\|_{X_1} \nonumber \\
		& \le \|g\|_{X_m} + 2m \|A_{a,1}f\|_{X_1} + m^2 \|f\|_{X_1} \,. \label{A.10b}
	\end{align}
\end{subequations}
Collecting \eqref{A.9} and \eqref{A.10} leads us to 
\begin{equation*}
	\|af\|_{X_{1,m}} \le (1+2m) \|A_{a,m}f\|_{X_{1,m}} + m^2 \|f\|_{X_{1,m}} \,,
\end{equation*}
 which in turn gives
 \begin{equation*}
 	\|A_{0,m}f\|_{X_{1,m}} \le \|A_{a,m}f\|_{X_{1,m}} + \|af\|_{X_{1,m}} \le 2(1+m) \|A_{a,m}f\|_{X_{1,m}} + m^2 \|f\|_{X_{1,m}}\,.
 \end{equation*}
Consequently, 
\begin{equation*}
	\frac{1}{4(1+m)} \left( \|A_{0,m}f\|_{X_{1,m}} + \|af\|_{X_{1,m}} \right) \le \|A_{a,m}f\|_{X_{1,m}} + \frac{m^2}{2(m+1)} \|f\|_{X_{1,m}}\,,
\end{equation*}
from which \eqref{A.8b} follows.
\end{proof}
	
\section{The Fragmentation-Diffusion Semigroup} \label{S.2.3}
	
We now consider the operator $\mathbb{A}_m=A_{a,m}+B_m$, where we recall that the nonlocal operator $B_{m}$ on $X_{1,m}$ is defined by 
\begin{equation*}
	\begin{split}
		\dom(B_m) & = \{ f\in X_{1,m}\ : \ af \in X_{1,m}\}\,, \\
		B_mf(x) & = \int_x^\infty a(y) b(x,y) f(y)\ \mathrm{d}y\,, \quad x\in (0,\infty)\,, \qquad f\in \dom(B_m)\,.
	\end{split}
\end{equation*}
We first, show that $B_m$ is $A_{a,m}$-bounded in $X_{1,m}$. 
	
\begin{lemma}\label{L.B.1}
Assume \eqref{A.0} and \eqref{B.0}. Let $m\ge 1$ and consider a measurable function $f$ on $(0,\infty)$ such that $af\in X_m$. Then 
\begin{equation}
	\int_0^\infty x^m \left| \int_x^\infty a(y) b(x,y) f(y)\ \mathrm{d}y \right|\ \mathrm{d}x \le \|af\|_{X_m}\,. \label{B.3}
\end{equation}
In addition,
\begin{equation}
M_1(B_mf)=M_1(af)\,,\qquad f\in \dom(B_m)\,, \label{M.200}
\end{equation}
and $B_m$ is $A_{a,m}$-bounded in $X_{1,m}$.
\end{lemma}
	
\begin{proof}
We infer from \eqref{B.0} and Fubini-Tonelli's theorem that \begin{align*}
	\int_0^\infty x^m \int_x^\infty a(y) b(x,y) |f(y)|\ \mathrm{d}y\mathrm{d}x & = \int_0^\infty a(y) |f(y)| \int_0^y x^m b(x,y)\ \mathrm{d}x\mathrm{d}y \\
	& \le \int_0^\infty y^{m-1} a(y) |f(y)| \int_0^y x b(x,y)\ \mathrm{d}x\mathrm{d}y = \|af\|_{X_m}\,,
\end{align*}
from which \eqref{B.3} readily follows. Next, \eqref{M.200} is a straightforward consequence of \eqref{B.0} and Fubini's theorem.

Finally, let $f\in \dom(A_{a,m}) \subset \dom(B_m)$. By \eqref{A.8b} and \eqref{B.3},
\begin{align*}
\|B_mf \|_{X_{1,m}} & \le \int_0^\infty (x+x^m)
\left| \int_x^\infty a(y) b(x,y) f(y)\ \mathrm{d}y \right|\ \mathrm{d}x \le \|a f\|_{X_{1,m}} \\
& \le 4 (m+1) \|A_{a,m}f\|_{X_{1,m}} + 4m(m+1) \|f\|_{X_{1,m}}\,,
\end{align*}
so that $B_m$ is $A_{a,m}$-bounded. 
\end{proof}
	
As already observed in the literature, see, e.g., \cite[Theorem~5.1.47~(c)]{BLL2020a}, the inequality \eqref{B.10} implies that, for each $m>1$, there is $\delta_m\in (0,1)$ such that 
\begin{equation}
	(1-\delta_m) y^m \ge \int_0^y x^m b(x,y)\ \mathrm{d}x\,, \qquad y\in (0,\infty)\,. \label{B.11}
\end{equation}

An immediate consequence of \eqref{B.11} is a strict domination of $af$ over $B_m f$ in $X_m$. 
	
\begin{lemma}\label{L.B.1.5}
Assume \eqref{A.0}, \eqref{B.0}, and \eqref{B.10}. Let $m>1$ and consider $f\in \dom(B_m)$. Then
\begin{equation*}
	\|B_m f\|_{X_m} \le (1-\delta_m)  \|af\|_{X_m}\,. 
\end{equation*}
\end{lemma}

\begin{proof}
It readily follows from \eqref{B.11} and Fubini's theorem that
\begin{align*}
	\|B_m f\|_{X_m} & \le \int_0^\infty x^m \int_x^\infty a(y) b(x,y) |f(y)|\ \mathrm{d}y\mathrm{d}x = \int_0^\infty a(y) |f(y)| \int_0^y x^m b(x,y)\ \mathrm{d}x\mathrm{d}y \\
	& \le (1-\delta_m)  \int_0^\infty y^m a(y) |f(y)| \mathrm{d}y =(1-\delta_m)  \|af\|_{X_m}\,.
\end{align*}	
\end{proof}

We shall see next that the property~\eqref{B.11} ensures that $B_m$ is a Miyadera perturbation of $A_{a,m}$. Recall that a similar result is available for the fragmentation equation without diffusion \cite{Ban2020}. 
	
\begin{proposition}\label{P.B.2}
	Let $m>1$ and assume \eqref{A.0}, \eqref{B.0}, and \eqref{B.10}. Then there are $q_m\in (0,1)$ and $t_m>0$ such that
	\begin{equation*}
	\int_0^{t_m} \| B_m e^{sA_{a,m}}f\|_{X_{1,m}}\ \mathrm{d}s \le  q_m \|f\|_{X_{1,m}}\,, \qquad f\in \dom(A_{a,m}) \,.
	\end{equation*}
	In particular, $B_m$ is a Miyadera perturbation of $A_{a,m}$.
\end{proposition}
	
\begin{proof}
Consider $f\in \dom(A_{a,m})$ and set $F(t) := e^{tA_{a,m}}f$ for $t\ge 0$. Owing to \Cref{P.A.1}, we have
\begin{equation}
	\|F(t)\|_{X_1} \le \|f\|_{X_1}\,, \quad \|F(t)\|_{X_{1,m}} \le e^{\omega_m t} \|f\|_{X_{1,m}}\,, \qquad t\ge 0\,. \label{B.12}
\end{equation}
In addition, $F$ is a classical solution to 
\begin{equation*}
	\frac{\mathrm{d}}{\mathrm{d}t} F - A_{a,m} F = 0\,, \quad t>0\,, \qquad F(0)=f\,,
\end{equation*}
and we deduce from \Cref{L.2.3} (with $\ell(x)=x^m$) that
\begin{equation*}
	\frac{\mathrm{d}}{\mathrm{d}t} \|F\|_{X_m} + \|a F\|_{X_m} \le m(m-1) \|F\|_{X_{m-2}}\,, \qquad t\ge 0\,.
\end{equation*}
Hence, after integration with respect to time,
\begin{equation}
	\int_0^t \|a F(s)\|_{X_m}\ \mathrm{d}s \le \|f\|_{X_m} + m(m-1) \int_0^t \|F(s)\|_{X_{m-2}}\ \mathrm{d}s\,, \qquad t\ge 0\,. \label{B.13}
\end{equation}
		
Now, let $t>0$. It follows from \eqref{B.0}, \eqref{B.3} (with $m=1$), and \Cref{L.B.1.5} that
\begin{equation}
	\int_0^t \| B_mF(s)\|_{X_{1,m}}\ \mathrm{d}s \le \int_0^t \|a F(s)\|_{X_1}\ \mathrm{d}s + (1-\delta_m) \int_0^t \|a F(s)\|_{X_m}\ \mathrm{d}s\,. \label{B.14}
\end{equation}
For $R>1$, we infer from \eqref{A.0} and \eqref{B.12} that
\begin{align*}
	\int_0^t \|a F(s)\|_{X_1}\ \mathrm{d}s & \le \|a\|_{L_\infty(0,R)} \int_0^t \int_0^R x |F(s,x)|\ \mathrm{d}x\mathrm{d}s + R^{1-m} \int_0^t \int_R^\infty x^m a(x) |F(s,x)|\ \mathrm{d}x\mathrm{d}s \\
	& \le \|a\|_{L_\infty(0,R)} \int_0^t \|F(s)\|_{X_1}\ \mathrm{d}s + R^{1-m} \int_0^t \int_0^\infty x^m a(x) |F(s,x)|\ \mathrm{d}x\mathrm{d}s \\
	& \le \|a\|_{L_\infty(0,R)} \|f\|_{X_1} t + R^{1-m} \int_0^t \| a F(s)\|_{X_m}\ \mathrm{d}s\,. 
\end{align*}
Combining \eqref{B.13}, \eqref{B.14}, and the above estimate with $R=R_m:=(\delta_m/2)^{1/(1-m)}$ gives
\begin{align}
	\int_0^t \| B_mF(s)\|_{X_{1,m}}\ \mathrm{d}s & \le \|a\|_{L_\infty(0,R_m)} \|f\|_{X_1} t + \left( 1 - \frac{\delta_m}{2} \right) \int_0^t \| a F(s)\|_{X_m}\ \mathrm{d}s \nonumber \\
	& \le \|a\|_{L_\infty(0,R_m)} \|f\|_{X_1} t + \left( 1 - \frac{\delta_m}{2} \right) \|f\|_{X_m} \nonumber \\
	& \qquad + m(m-1)\left( 1 - \frac{\delta_m}{2} \right) \int_0^t \|F(s)\|_{X_{m-2}}\ \mathrm{d}s \nonumber \\
	& \le \|a\|_{L_\infty(0,R_m)} \|f\|_{X_1} t + \left( 1 - \frac{\delta_m}{2} \right) \|f\|_{X_m} \nonumber \\
	& \qquad + m(m-1) \int_0^t \|F(s)\|_{X_{m-2}}\ \mathrm{d}s\,. \label{B.15} 
\end{align}
		
At this point, we handle the cases $m\ge 3$ and $m\in (1,3)$ in a different way. We first consider $m\ge 3$. We use Young's inequality, along with \eqref{B.12}, to obtain
\begin{align}
	m(m-1) \int_0^t \|F(s)\|_{X_{m-2}}\ \mathrm{d}s & \le m(m-1) \int_0^t \left[ \frac{m-3}{m-1} \|F(s)\|_{X_m} + \frac{2}{m-1} \|F(s)\|_{X_1} \right]\ \mathrm{d}s \nonumber \\
	& \le \frac{m(m-3)}{1+\omega_m} \|f\|_{X_{1,m}} \left( e^{(1+\omega_m)t} - 1 \right) + 2m \|f\|_{X_1} t\,. \label{B.16}
\end{align}
Collecting \eqref{B.15} and \eqref{B.16} leads us to
\begin{align*}
	\int_0^t \| B_mF(s)\|_{X_{1,m}}\ \mathrm{d}s & \le \left[ \|a\|_{L_\infty(0,R_m)} t + \frac{m(m-3)}{1+\omega_m} \left( e^{(1+\omega_m)t} - 1 \right) + 2m t \right] \|f\|_{X_1} \\
	& \qquad + \left[ 1 - \frac{\delta_m}{2} + \frac{m(m-3)}{1+\omega_m} \left( e^{(1+\omega_m)t} - 1 \right) \right] \|f\|_{X_m}\,.
\end{align*}
We now pick $t_m>0$ such that 
\begin{equation*}
	\left( \|a\|_{L_\infty(0,R_m)} + 2m \right) t_m \le 1 - \frac{\delta_m}{2} \quad \text{ and } \quad \frac{m(m-3)}{1+\omega_m} \left( e^{(1+\omega_m)t_m} - 1 \right)\le \frac{\delta_m}{4}
\end{equation*}
and infer from the previous estimate (with $t=t_m$) that
\begin{equation*}
	\int_0^{t_m} \| B_mF(s)\|_{X_{1,m}}\ \mathrm{d}s \le \left( 1 - \frac{\delta_m}{4} \right) \|f\|_{X_{1,m}}\,.
\end{equation*}
Recalling that $B_m$ is $A_{a,m}$-bounded by \Cref{L.B.1}, we have thus established that $B_m$ is a Miyadera perturbation of $A_{a,m}$ for $m\ge 3$.
		
Let us now consider $m\in (1,3)$. In that case, $m-2\in (-1,1)$ and it follows from \Cref{L.2.1}, \Cref{L.A.2}, and \eqref{B.12} that, for $s\in (0,t)$,  
\begin{align*}
	\|F(s)\|_{X_{m-2}} & \le \frac{2(3-m)^{(m-3)/2}}{m-1} \|F''(s)\|_{X_1}^{(3-m)/2} \|F(s)\|_{X_1}^{(m-1)/2} \\
	& \le \frac{6(3-m)^{(m-3)/2}}{m-1} \|A_{a,1}F(s)\|_{X_1}^{(3-m)/2} \|f\|_{X_1}^{(m-1)/2}\,.
\end{align*}
Owing to the analyticity of $\left( e^{tA_{a,m}} \right)_{t\ge 0}$, see \Cref{P.A.1}, we further infer from \cite[Theorem~2.5.2]{Paz1983} that there is $C>0$ such that
	\begin{equation*}
		\|A_{a,1} e^{s A_{a,1}} \|_{\mathcal{L}(X_1)} \le C \frac{e^s}{s} \le C \frac{e^t}{s}\,, \qquad s\in (0,t)\,.
	\end{equation*}
Combining the above two estimates gives
\begin{equation*}
	\|F(s)\|_{X_{m-2}} \le C(m) \|f\|_{X_1} e^{(3-m)t/2} s^{(m-3)/2}\,, \qquad s\in (0,t)\,.
\end{equation*}
Hence, recalling \eqref{B.15}, 
\begin{align*}
	\int_0^t \| B_m F(s)\|_{X_{1,m}}\ \mathrm{d}s & \le \left[ \|a\|_{L_\infty(0,R_m)} t + C(m) e^{(3-m)t/2} t^{(m-1)/2} \right] \|f\|_{X_1} \\
	& \qquad  + \left( 1 - \frac{\delta_m}{2} \right) \|f\|_{X_m} \,.
\end{align*}
We now choose $t_m>0$ such that 
\begin{equation*}
	\|a\|_{L_\infty(0,R_m)} t_m + C(m) e^{(3-m)t_m/2} t_m^{(m-1)/2} \le 1 - \frac{\delta_m}{2}
\end{equation*}
and deduce from the previous inequality (with $t=t_m$) that
\begin{equation*}
	\int_0^{t_m} \| B_m F(s)\|_{X_{1,m}}\ \mathrm{d}s \le  \left( 1 - \frac{\delta_m}{2} \right) \|f\|_{X_{1,m}} \,.
\end{equation*}
Consequently, using again \Cref{L.B.1}, $B_m$ is also a Miyadera perturbation of $A_{a,m}$ when $m\in (1,3)$. 
\end{proof}

We are now in a position to prove the first two statements in \Cref{T.1} for the operator $\mathbb{A}_m=A_{a,m}+B_m$:

\begin{proof}[Proof of \Cref{T.1}~(a)-(b)] We handle the cases $m=1$ and $m>1$ separately.

\medskip

\noindent{(a)}. If $m=1$, then $A_{a,1}\in \mathcal{G}_+(X_1,1,0)$ by \Cref{P.A.1}, so that it generates a substochastic semigroup in $X_1$. Moreover, $\dom(A_{a,1})\subset \dom(B_1)$ and $B_1$ is obviously positive due to the nonnegativity of $a$ and $b$. Also, for $f\in \dom(A_{a,1})$,
\begin{equation*}
M_1(A_{a,1}f+B_1f) = - \int_0^\infty f'(x)\ \mathrm{d}x - M_1(af) + M_1(B_1f) = 0
\end{equation*}
by \Cref{L.2.1}, \eqref{M.200}, and the Dirichlet boundary condition. Consequently, we infer from \cite{Voi1987} and \cite[Theorem~4.9.16]{BLL2020a} that there is an extension $\tilde{\mathbb{A}}_1\in \mathcal{G}_+(X_1,1,0)$ of $\mathbb{A}_1$. 

\medskip

\noindent{(b)}. Let $m>1$. Since $A_{a,m}\in\mathcal{H}(X_{1,m})$ by \Cref{P.A.1} and $B_m$ is a Miyadera perturbation of $A_{a,m}$ by \Cref{P.B.2}, it follows from \cite[Corollary~III.3.16 \& Exercise~III.3.17]{EnNa2000} that $\mathbb{A}_m=A_{a,m}+B_m\in \mathcal{H}(X_{1,m})$ with  $\dom(\mathbb{A}_m) = \dom(A_{a,m})$. Note that $\D(A_{a,m})$  and $\D(\mathbb{A}_m)$ are both Banach spaces and that $\D(A_{a,m})$ is continuously embedded in $\D(\mathbb{A}_m)$, since $B_m$ is $A_{a,m}$-bounded in $X_{1,m}$ according to  \Cref{L.B.1}. Consequently, $\D(\mathbb{A}_m)\doteq \D(A_{a,m})$ by the open mapping theorem.

We now check the positivity of $\left( e^{t\mathbb{A}_m} \right)_{t\ge 0}$, bearing in mind that we already know from \Cref{P.A.1} that $A_{a,m}$ is resolvent positive. Pick $\lambda>0$ sufficiently large. Then $\lambda - \mathbb{A}_m$ is invertible with inverse given by 
\begin{align*}
(\lambda-\mathbb{A}_m)^{-1}&=(\lambda-A_{a,m}-B_m)^{-1}  = (\lambda-A_{a,m})^{-1} \left( 1 - B_m (\lambda - A_{a,m})^{-1} \right)^{-1} \\
& = (\lambda-A_{a,m})^{-1} \sum_{j=0}^\infty \left[ B_m (\lambda - A_{a,m})^{-1} \right]^j\,,
\end{align*}
where the Neumann series converges since $B_m$ is a Miyadera perturbation of $A_{a,m}$, see the proof of \cite[Theorem~III.3.14]{EnNa2000}. Now, $B_m$ is obviously a positive operator on $X_{1,m}$ due to the non-negativity of $a$ and $b$, and the positivity of $(\lambda-\mathbb{A}_m)^{-1}$ directly follows from the above identity.

Finally, as in the proof of~(a), we have $M_1(\mathbb{A}_m f) = 0$
for any $f\in \dom(A_{a,m})$ by \Cref{L.2.1}, \eqref{M.200}, and the Dirichlet boundary condition, so that \eqref{M.100} immediately follows.
\end{proof}

\begin{proof}[Proof of \Cref{P.10}~(a)]
Let $m\ge 1$. The operator $A_{a,m}$ belongs to $\mathcal{G}_+(X_{1,m})\cap \mathcal{H}(X_{1,m})$ by \Cref{P.A.1}. Since $a\in L_\infty(0,\infty)$ and $b$ satisfies \eqref{B.0}, the operator $B_m$ is a positive bounded operator on $X_{1,m}$. On the one hand, it now follows from well-known perturbation results that $\mathbb{A}_m=A_{a,m}+B_m$ belongs to $\mathcal{H}(X_{1,m})$, see  \cite[Theorem~3.2.1]{Paz1983}. On the other hand, the same argument as in the proof of \Cref{T.1}~(b) ensures the positivity of $\left( e^{t\mathbb{A}_m} \right)_{t\ge 0}$. Finally, for $m=1$, it readily follows from \cite[Proposition~4.9.16]{BLL2020a} that $\mathbb{A}_1=\tilde{\mathbb{A}}_1\in \mathcal{G}_+(X_1,1,0)$, thereby completing the proof. 
\end{proof}

\section{Immediate Compactness of the Semigroup} \label{S.2.4}

We now turn to compactness properties of the semigroup  $(e^{t\mathbb{A}_m})_{t\ge 0}$ for $m>1$ as stated in \Cref{T.1}~(c). To avoid loss of compactness for large sizes, we further require $a$ to diverge to infinity for large sizes, thus excluding bounded overall fragmentation rates.
	
\begin{lemma}\label{L.C.1}
	Let $m\ge 1$ and assume that $a$ satisfies \eqref{A.0} and \eqref{C.0}. Then $\D(A_{a,m})\doteq \D(\mathbb{A}_m)$ is compactly embedded in $X_{1,m}$.
\end{lemma}

\begin{proof}
Recall that the relation $\D(A_{a,m})\doteq \D(\mathbb{A}_m)$ is established in the proof of \Cref{T.1}~(b).	Let $(f_n)_{n\ge 1}$ be a bounded sequence in $\D (A_{a,m})$. According to \Cref{L.2.1} and \Cref{L.A.2}, there is $C>0$ such that
	\begin{subequations}\label{C.1}
	\begin{align}
	\sup_{x\ge 0} \left\{ |f_n(x)| + x|f_n'(x)| \right\} & \le C\,, \qquad n\ge 1\,, \label{C.1a} \\
	\|f_n\|_{X_{1,m}}  + \|f_n''\|_{X_1} + \|a f_n\|_{X_m} & \le C\,, \qquad n\ge 1\,. \label{C.1b}
	\end{align}
\end{subequations}
On the one hand, we infer from \eqref{C.1a} and Arzel\`a-Ascoli's theorem that $(f_n)_{n\ge 1}$ is relatively compact in $C([1/R,R])$ for each $R>1$. There are thus a subsequence $(f_{n_j})_{j\ge 1}$ and $f\in C((0,\infty))$ such that
\begin{equation} 
	\lim_{j\to\infty} f_{n_j}(x) = f(x)\,, \qquad x\in (0,\infty)\,. \label{C.2}
\end{equation}
On the other hand, it follows from \eqref{C.1} that, if $R>1$ and $E$ is a measurable subset of $(0,\infty)$, then, for $n\ge 1$,
\begin{align}
	\int_E (x+x^m) |f_n(x)|\ \mathrm{d}x & \le \int_E (x+x^m) \mathbf{1}_{(0,R)}(x) |f_n(x)|\ \mathrm{d}x + \int_R^\infty (x+x^m) |f_n(x)|\ \mathrm{d}x \nonumber \\
	& \le (R+R^m) \|f_n\|_{L_\infty(0,\infty)} |E\cap (0,R)| \nonumber \\
	& \qquad + \dfrac{2}{\inf_{x\ge R} \{a(x)\}} \int_R^\infty x^m\ a(x) |f_n(x)|\ \mathrm{d}x \nonumber \\
	& \le 2 C R^m |E\cap (0,R)| + \dfrac{2C}{\inf_{x\ge R} \{a(x)\}}\,. \label{C.3}
\end{align}
A first consequence of \eqref{C.3} with $E=(R,\infty)$ is that
\begin{equation*}
	\sup_{n\ge 1} \int_R^\infty (x+x^m) |f_n(x)|\ \mathrm{d}x \le \dfrac{2C}{\inf_{x\ge R} \{a(x)\}}\,,
\end{equation*}
from which we deduce by \eqref{C.0} that
\begin{equation}
	\lim_{R\to\infty} \sup_{n\ge 1} \int_R^\infty (x+x^m) |f_n(x)|\ \mathrm{d}x = 0\,. \label{C.4}
\end{equation}
We next infer from \eqref{C.3} that, for $\delta>0$,
\begin{equation*}
	\eta(\delta) := \sup\left\{ \int_E (x+x^m) |f_n(x)|\ \mathrm{d}x\ :\ n \ge 1\,, \ E\in \mathcal{B}((0,\infty))\,,  \ |E|\le\delta \right\}
\end{equation*}
satisfies
\begin{equation*}
	\eta(\delta) \le 2 C R^m \delta + \dfrac{2C}{\inf_{x\ge R} \{a(x)\}}
\end{equation*}
for all $R>1$. We first let $\delta\to 0$ and then $R\to\infty$ in the above estimate and use once more \eqref{C.0} to conclude that
\begin{equation}
	\lim_{\delta\to 0} \eta(\delta) = 0\,. \label{C.5}
\end{equation}
Gathering \eqref{C.4} and \eqref{C.5} implies that the sequence $(f_n)_{n\ge 1}$ is uniformly integrable in $X_{1,m}$ and thus weakly compact in $X_{1,m}$ by Dunford-Pettis' theorem. This just established weak compactness in $X_{1,m}$, along with the pointwise convergence \eqref{C.2} and Vitali's theorem, entails that $(f_{n_j})_{j\ge 1}$ converges to $f$ in $X_{1,m}$, thereby completing the proof.
\end{proof}

We are now in a position to finish off the proof of \Cref{T.1}.

\begin{proof}[Proof of \Cref{T.1}~(c)]
Let $m>1$. By \Cref{L.C.1}, $(\lambda - \mathbb{A}_m)^{-1}$ is compact for $\lambda>0$ large enough and, since $m>1$, the analyticity of $\left( e^{t\mathbb{A}_m} \right)_{t\ge 0}$ implies that it is continuous with respect to the operator norm for positive times \cite[Lemma~2.4.2]{Paz1983}. We now may apply \cite[Theorem~2.3.3]{Paz1983} to conclude that $\left( e^{t\mathbb{A}_m} \right)_{t\ge 0}$ is immediately compact.
\end{proof}

The compactness result of \Cref{L.C.1} is not valid under the sole assumption~\eqref{A.0} on $a$. 
In particular, we show that it fails when $a$ is bounded.

\begin{lemma}\label{L.C.2}
    Let $m\ge 1$ and assume that $a\in L_\infty(0,\infty)$. Then the embedding of  $\D(A_{a,m})\doteq D(\mathbb{A}_m)$ in $X_{1,m}$ is not compact.
\end{lemma}

\begin{proof}
    Let $\varphi\in C_c^\infty(\mathbb{R})$ be such that $0\le \varphi\le 1$, $\mathrm{supp}\ \varphi\subset [-1,1]$, and $\|\varphi\|_{L_1(\mathbb{R})}=1$. We fix $m\ge 1$ and set
    \begin{equation*}
        \varphi_n(x) := \frac{1}{n+n^m} \varphi(x-n)\,, \qquad x\in (0,\infty)\,, \quad n\ge 1\,.
    \end{equation*}
Straightforward computations show that
\begin{align*}
    \|\varphi_n\|_{X_{1,m}} + \|a \varphi_n\|_{X_{1,m}} & \le \left( 1 + \|a\|_{L_\infty(0,\infty)} \right) \|\varphi_n\|_{X_{1,m}} \\
    & \le \left( 1 + \|a\|_{L_\infty(0,\infty)} \right) \left( 1 + \int_{-1}^1 \frac{1+m2^m n^{m-1}}{n+n^m} \varphi(y)\ \mathrm{d}y \right) \\
    & \le (2 + m 2^m) \left( 1 + \|a\|_{L_\infty(0,\infty)} \right)
\end{align*}
and
\begin{align*}
    \|\varphi_n''\|_{X_{1,m}} \le \frac{n+1 + (n+1)^m}{n+n^m} \int_{-1}^1 |\varphi''(y)|\ \mathrm{d}y \le 2^m \|\varphi''\|_{L_1(\mathbb{R})}\,,
\end{align*}
as well as
\begin{equation}
    \lim_{n\to\infty} \|\varphi_n\|_{X_{1,m}} = 1\,, \qquad \lim_{n\to\infty} \|\varphi_n\|_{L_\infty(0,\infty)} = 0\,. \label{NC.1}
\end{equation}
Therefore, the sequence $(\varphi_n)_{n\ge 1}$ is bounded in $\D(A_{a,m})$ but cannot converge in $X_{1,m}$ due to \eqref{NC.1}.
\end{proof}

\begin{proof}[Proof of \Cref{P.10}~(b)]
Let $m\ge 1$ and $a\in L_\infty(0,\infty)$. Since $\D(\mathbb{A}_m)$ is not compactly embedded in $X_{1,m}$ by \Cref{L.C.2}, the resolvent $(\lambda-\mathbb{A}_m)^{-1}$ is not compact. Hence, \cite[Theorem 2.3.3]{Paz1983} implies that the semigroup $(e^{t\mathbb{A}_m})_{t\ge 0}$ is not compact.
\end{proof}

\section{Steady States and Convergence} \label{S.2.5}

Throughout this section, we assume that $a$ and $b$ satisfy \eqref{A.0}, \eqref{B.0}, \eqref{B.10}, and \eqref{C.0}, and that $a>0$ and $b>0$. 

We begin with the construction of stationary solutions with the help of Schauder's fixed point theorem.

\begin{lemma}\label{L.100}
There is a unique nonnegative
\begin{equation*}
\psi_1\in \bigcap_{r\ge 1} \dom(\mathbb{A}_r)
\end{equation*} 
such that $M_1(\psi_1)=1$ and $\mathrm{ker}(\mathbb{A}_m) = \mathrm{ker}(\mathbb{A}_m^2) = \mathbb{R}\psi_1$ for all $m\ge 1$. 
\end{lemma}

\begin{proof} We split the proof into three steps.

\medskip

\noindent\textbf{Step~1.}  The uniqueness of a solution $\psi\in \dom(\mathbb{A}_1)$ to $\mathbb{A}_1\psi=0$ satisfying $M_1(\psi)=1$ relies on the dissipativity properties of $\mathbb{A}_1$ in $X_1$ and can be shown exactly as in the proofs of \cite[Lemma~3.5]{EMRR2005} and \cite[Proposition~3]{Lau2004}, to which we refer. 

\medskip

\noindent\textbf{Step~2.} We now turn to the existence part. 
Let $m\ge 3$ and consider $f\in X_{1,m}^+$ satisfying $M_1(f)=1$. Setting $F(t) := e^{t\mathbb{A}_m}f$ for $t\ge 0$, it readily follows from \Cref{T.1}~(b) that
\begin{equation}
F(t)\ge 0 \;\text{ and }\; M_1(F(t))=1\,, \qquad t\ge 0\,. \label{Z.1}
\end{equation}
Next, by \eqref{B.11} and Fubini's theorem, 
\begin{align*}
\frac{\mathrm{d}}{\mathrm{d}t} M_m(F(t)) & = - m \int_0^\infty x^{m-1} \partial_x F(t,x)\ \mathrm{d}x - \int_0^\infty x^m a(x) F(t,x)\ \mathrm{d}x \\
& \qquad + \int_0^\infty a(y) F(t,y) \int_0^y x^m b(x,y)\ \mathrm{d}x\mathrm{d}y \\
& \le m(m-1) M_{m-2}(F(t)) - \delta_m M_m(aF(t))\,.
\end{align*}
Owing to \eqref{C.0}, there is $x_*>0$ such that $a(x)\ge 1$ for $x\ge x_*$. Consequently, using \eqref{Z.1},
\begin{align*}
\frac{\mathrm{d}}{\mathrm{d}t} M_m(F(t)) + \delta_m M_m(F(t)) & \le \frac{\mathrm{d}}{\mathrm{d}t} M_m(F(t)) + \delta_m \int_{x_*}^\infty x^m F(t,x)\ \mathrm{d}x + \delta_m \int_0^{x_*} x^m F(t,x)\ \mathrm{d}x \\
& \le \frac{\mathrm{d}}{\mathrm{d}t} M_m(F(t)) + \delta_m M_m(aF(t)) + \delta_m x_*^{m-1} \int_0^{x_*} x F(t,x)\ \mathrm{d}x \\
& \le m(m-1) M_{m-2}(F(t)) + \delta_m x_*^{m-1} \,.
\end{align*}
Since $m\ge 3$, we now deduce from Young's inequality that
\begin{align*}
\frac{\mathrm{d}}{\mathrm{d}t} M_m(F(t)) + \delta_m M_m(F(t)) & \le \frac{\delta_m}{2} M_m(F(t)) + 2m \left( \frac{2m(m-3)}{\delta_m} \right)^{(m-3)/2} M_1(F(t)) \\
& \qquad + \delta_m x_*^{m-1} \,.
\end{align*}
Hence, by \eqref{Z.1},
\begin{equation*}
\frac{\mathrm{d}}{\mathrm{d}t} M_m(F(t)) + \frac{\delta_m}{2} M_m(F(t)) \le   2m \left( \frac{2m(m-3)}{\delta_m} \right)^{(m-3)/2} + \delta_m x_*^{m-1}=: \frac{\delta_m \mu_m}{2}\,, \qquad t\ge 0\,.
\end{equation*}
After integration with respect to time, we conclude that
\begin{equation}
M_m(F(t)) \le \max\left\{ M_m(f) , \mu_m \right\}\,, \qquad t\ge 0\,. \label{Z.2}
\end{equation}
Now, introducing
\begin{equation*}
\mathcal{C}_m := \{ f\in X_{1,m}^+\ :\ M_1(f)=1\,, \ M_m(f) \le \mu_m\}\,, 
\end{equation*}
which is  a closed convex subset of $X_{1,m}$,
an immediate consequence of \eqref{Z.1} and \eqref{Z.2} is that 
\begin{equation*}
e^{t\mathbb{A}_m}\mathcal{C}_m \subset \mathcal{C}_m\,, \qquad t\ge 0\,.
\end{equation*}
Owing to the compactness of $e^{t\mathbb{A}_m}$ in $X_{1,m}$ for all $t>0$, see \Cref{T.1}~(c), we argue as in the proofs of \cite[Theorem~22.13]{Ama1990} and \cite[Theorem~5.2]{GPV2004} to deduce from Schauder's fixed point theorem that there is $\psi_m\in \mathcal{C}_m$ such that
\begin{equation*}
e^{t\mathbb{A}_m}\psi_m = \psi_m\,, \qquad t\ge 0\,.
\end{equation*}
Equivalently, $\mathbb{A}_m\psi_m=0$ and we have thus shown  the existence of a stationary solution to \eqref{FD.1} for $m\ge 3$. 
Obviously, $\psi_3$ also belongs to $\dom(\mathbb{A}_m)$ and satisfies $\mathbb{A}_m\psi_3=0$  for any $m\in [1,3)$. Thus, there is at least one stationary solution $\psi_m$ to \eqref{FD.0} for any $m\ge 1$. Obviously, $\psi_m\in \dom(\mathbb{A}_1)$ solves $\mathbb{A}_1\psi_m=0$ for every $m\ge 1$, and we infer from \textbf{Step~1} that $\psi_m=\psi_1$ for every $m\ge 1$.

\medskip

\noindent\textbf{Step~3.} We finally identify $\mathrm{ker}(\mathbb{A}_m^2)$. To this end, let $f\in \mathrm{ker}(\mathbb{A}_m^2)$. Then $\mathbb{A}_m f$ belongs to $\mathrm{ker}(\mathbb{A}_m)$, so that \textbf{Step~2} implies that there is $\mu\in\mathbb{C}$ such that $\mathbb{A}_m f = \mu \psi_1$. Therefore,
\begin{equation*}
\mu = \mu M_1(\psi_1) = M_1(\mathbb{A}_m f)=0\,.
\end{equation*}
Hence, $f\in \mathrm{ker}(\mathbb{A}_m)$.
\end{proof}

We now supply refined information on the spectrum of $\mathbb{A}_m$ for $m>1$. 

\begin{lemma}\label{L.101}
Let $m>1$. The spectrum $\sigma(\mathbb{A}_m)$ of $\mathbb{A}_m$ only consists of isolated eigenvalues and satisfies 
\begin{equation}
\sigma(\mathbb{A}_m) \subset \{0\} \cup \{ \lambda\in \mathbb{C}\ :\ \mathsf{Re}\, \lambda < -\varepsilon_m \} \label{Sp.2}
\end{equation}
for some $\varepsilon_m>0$. Moreover, $s(\mathbb{A}_m)=0$ is a simple eigenvalue of $\mathbb{A}_m$. 
\end{lemma}

\begin{proof}
Owing to the immediate compactness of $(e^{t\mathbb{A}_m})_{t\ge 0}$, see \Cref{T.1}~(c), and \cite[Corollary~V.3.2]{EnNa2000}, the spectrum $\sigma(\mathbb{A}_m)$ only consists of isolated eigenvalues which are poles of the resolvent with finite algebraic multiplicity. Moreover, for any $r\in\mathbb{R}$, 
\begin{equation}
\#\{\lambda\in \sigma(\mathbb{A}_m)\ :\ \mathsf{Re}\,\lambda \ge r\} < \infty\,. \label{Sp.0}
\end{equation}

We next claim that $s(\mathbb{A}_m)=0$. Indeed, since $\mathbb{A}_m \subset \mathbb{A}_1 \subset \tilde{\mathbb{A}}_1$ and $\tilde{\mathbb{A}}_1\in \mathcal{G}(X_1,1,0)$ by \Cref{T.1}, any eigenvalue of $\mathbb{A}_m$ is also an eigenvalue of $\tilde{\mathbb{A}}_1$ and it follows from \cite[Corollary~1.3.6]{Paz1983} that 
$$
\{ \lambda\in\mathbb{C}\ :\ \mathsf{Re}\, \lambda > 0 \} \subset \rho(\tilde{\mathbb{A}}_1) \,. 
$$
Consequently, any eigenvalue of $\mathbb{A}_m$ has a non-positive real part. Thus, 
\begin{equation}\label{Sp.1}
\sigma(\mathbb{A}_m)\subset \{ \lambda\in\mathbb{C}\ :\ \mathsf{Re}\, \lambda \le 0 \}\,.
\end{equation}
Since zero belongs to the spectrum of $\mathbb{A}_m$ by \Cref{L.100}, we deduce from \eqref{Sp.1} that $s(\mathbb{A}_m)=0$ is a pole of the resolvent of $\mathbb{A}_m$. Recalling that $(e^{t\mathbb{A}_m})_{t\ge 0}$ is a positive semigroup on the Banach lattice $X_{1,m}$, it follows from \cite[Theorem~8.14]{CHADP1987} that $\sigma(\mathbb{A}_m)\cap i\mathbb{R}$ is either reduced to $\{0\}$ or contains infinitely many elements. The latter being ruled out by \eqref{Sp.0}, we conclude that $\sigma(\mathbb{A}_m)\cap i\mathbb{R} = \{0\}$. Since all eigenvalues are isolated, this last property ensures that there is $\varepsilon_m>0$ such that \eqref{Sp.2} holds true. 

Finally, since $\mathrm{ker}(\mathbb{A}_m) = \mathrm{ker}(\mathbb{A}_m^2) = \mathbb{R}\psi_1$ by \Cref{L.100}, zero is a simple eigenvalue of $\mathbb{A}_m$ according to \cite[Section~IV.1.17]{EnNa2000}. 
\end{proof}

\begin{proof}[Proof of \Cref{T.2}]
Let $m>1$. From \Cref{L.100} we obtain the existence of a
unique nonnegative
\begin{equation*}
\psi_1\in \bigcap_{r\ge 1} \dom(\mathbb{A}_r)
\end{equation*} 
such that $M_1(\psi_1)=1$ and $\mathrm{ker}(\mathbb{A}_m) = \mathbb{R}\psi_1$. We next infer from \Cref{L.101} that zero is a dominant eigenvalue of $\mathbb{A}_m$ and a first-order pole of its resolvent with residue $P$, where $P\in\mathcal{L}(X_{1,m})$ denotes the spectral projection onto $\mathrm{ker}(\mathbb{A}_m)$ and is given by
\begin{equation}\label{x1}
Pf = \lim_{\lambda\to 0} \lambda (\lambda - \mathbb{A}_m)^{-1}f\,, \qquad f\in X_{1,m}\,, 
\end{equation}
see, e.g., \cite[Section~IV.1.17]{EnNa2000}. It then follows from  \cite[Corollary~V.3.3]{EnNa2000} that there are $N_m\ge 1$ and $\nu_m>0$ such that
\begin{equation}
\| e^{t\mathbb{A}_m} - P \|_{\mathcal{L}(X_{1,m})} \le N_m e^{-\nu_m t}\,, \qquad t\ge 0\,. \label{Co.1}
\end{equation}

It only remains to identify the spectral projection $P$. Introducing $g_\lambda := \lambda (\lambda - \mathbb{A}_m)^{-1}f$ for $f\in X_{1,m}$, we have
\begin{equation*}
\lambda f = \lambda g_\lambda - \mathbb{A}_m g_\lambda\,,
\end{equation*}
from which we readily deduce that $M_1(f) = M_1(g_\lambda)$. Therefore, \eqref{x1} implies $M_1(Pf) = M_1(f)$. Since $Pf\in \mathbb{R}\psi_1$ and $M_1(\psi_1)=1$, we conclude that 
$$ 
P f=M_1(f)\psi_1\,,\quad f\in X_{1,m}\,.
$$
Recalling \eqref{Co.1}, the above identity completes the proof of \Cref{T.2}.
\end{proof}

\section{Stationary Solutions Revisited}\label{sec.SM}

We now prove the existence of  a stationary solution to \eqref{FD.0} when the overall fragmentation rate $a$ may be bounded for large sizes but does not decay to zero. Specifically, we assume that $a$ satisfies \eqref{St.1}; that is,
\begin{equation*}
	\alpha := \frac{1}{2}\liminf_{x\to\infty} a(x) \in (0,\infty)\,. 
\end{equation*}

\begin{proof}[Proof of \Cref{P.3}]
As in \Cref{L.100}, the proof of the uniqueness assertion in \Cref{P.3} relies on the dissipativity properties of $\mathbb{A}_1$ in $X_1$ and can be shown exactly as in the proofs of \cite[Lemma~3.5]{EMRR2005} and \cite[Proposition~3]{Lau2004}, to which we refer. 
	
As for the existence assertion, we employ a compactness method. Let $n\ge 1$. We set $a_n(x) := a(x) +x/n$ for $x>0$ and, for $m\ge 1$, we denote the operators $B_m$ and $\mathbb{A}_m$ with $a_n$ instead of $a$ by $B_{m,n}$ and $\mathbb{A}_{m,n}$, respectively. Since $a_n(x)\to \infty$ as $x\to\infty$, we infer from \Cref{L.100} that there is a unique nonnegative
\begin{equation*}
	\psi_{1,n}\in \bigcap_{r\ge 1} \dom(\mathbb{A}_{r,n})
\end{equation*} 
such that $M_1(\psi_{1,n})=1$ and $\mathrm{ker}(\mathbb{A}_{m,n}) = \mathbb{R}\psi_{1,n}$ for all $m\ge 1$. In particular, given $m>3$, the function $\psi_{1,n}$ belongs to $\dom(A_{0,m})$ with $a_n \psi_{1,n}\in X_{1,m}$ and solves 
\begin{equation}
	- \psi_{1,n}'' + a_n \psi_{1,n} = B_{m,n}\psi_{1,n} \;\text{ in }\; (0,\infty)\,, \qquad \psi_{1,n}(0)=0\,. \label{St.2}
\end{equation}
It follows from \eqref{St.2}, \Cref{L.B.1.5}, and Young's inequality that, for $\varepsilon>0$, 
\begin{align*}
	M_m(a_n \psi_{1,n}) & = M_m(B_{m,n}\psi_{1,n}) - m \int_0^\infty x^{m-1} \psi_{1,n}'(x)\ \mathrm{d}x \\
	& \le (1-\delta_m) M_m(a_n \psi_{1,n}) + m(m-1) M_{m-2}(\psi_{1,n}) \\
	& \le (1-\delta_m) M_m(a_n \psi_{1,n}) + m(m-3) \varepsilon M_{m}(\psi_{1,n}) + 2m \varepsilon^{(3-m)/2} M_1(\psi_{1,n})\,. 
\end{align*}
Hence,
\begin{equation}
	\delta_m M_m(a_n \psi_{1,n}) \le m(m-3) \varepsilon M_{m}(\psi_{1,n}) + 2m \varepsilon^{(3-m)/2}\,. \label{St.3}
\end{equation}
Owing to \eqref{St.1}, there is $x_*>0$ such that
\begin{equation}
	a(x) \ge \alpha\,, \qquad x\ge x_*\,. \label{St.4}
\end{equation}
In view of \eqref{St.3} and \eqref{St.4}, we obtain
\begin{align*}
	\alpha \delta_m M_m(\psi_{1,n}) & \le \alpha\delta_m x_*^{m-1} \int_0^{x_*} x \psi_{1,n}(x)\ \mathrm{d}x + \delta_m \int_{x_*}^\infty x^m a_n(x) \psi_{1,n}(x)\ \mathrm{d}x \\
	& \le \alpha\delta_m x_*^{m-1} M_1(\psi_{1,n}) + m(m-3) \varepsilon M_{m}(\psi_{1,n}) + 2m \varepsilon^{(3-m)/2}\,.
\end{align*}
Choosing $\varepsilon=\alpha\delta_m/2m(m-3)$ in the above inequality gives
\begin{equation*}
	\frac{\alpha \delta_m}{2} M_m(\psi_{1,n}) \le \alpha\delta_m x_*^{m-1} + 2m \left( \frac{\alpha\delta_m}{2m(m-3)} \right)^{(3-m)/2}\,.
\end{equation*}
Therefore, there is a positive constant $c_1(m)$ depending only on $a$ and $m$ such that
\begin{equation}
	M_m(\psi_{1,n}) \le c_1(m)\,, \qquad n\ge 1\,. \label{St.5}
\end{equation}
Several additional estimates can now be derived from \eqref{St.5}. Indeed, it readily follows from \eqref{St.2}, \eqref{St.3} (with $\varepsilon=1$), and \Cref{L.B.1.5} that, for $n\ge 1$, 
\begin{align}
	\|\psi_{1,n}''\|_{X_m} + \|a_n \psi_{1,n}\|_{X_m} + \| B_{m,n}\psi_{1,n}\|_{X_m} & \le 2 \left( \|a_n \psi_{1,n}\|_{X_m} + \| B_{m,n}\psi_{1,n}\|_{X_m} \right) \nonumber \\
	& \le 2(2-\delta_m) \|a_n \psi_{1,n}\|_{X_m} \le 4 M_m(a_n \psi_{1,n}) \nonumber \\
	& \le  \frac{4m(m-3)c_1(m)+8m}{\delta_m} =:c_2(m)\,. \label{St.6}
\end{align}
Similarly, by \eqref{A.0}, \eqref{St.2}, \eqref{St.5} (with $m=4$), and \Cref{L.B.1},
\begin{align}
	& \|\psi_{1,n}''\|_{X_1} + \|a_n \psi_{1,n}\|_{X_1} + \| B_{m,n}\psi_{1,n}\|_{X_1} \nonumber\\ 
	& \qquad\qquad \le 2 \left( \|a_n \psi_{1,n}\|_{X_1} + \| B_{m,n}\psi_{1,n}\|_{X_1} \right) \le 4 \|a_n \psi_{1,n}\|_{X_m} \nonumber \\
	&  \qquad\qquad \le 4 \left( 1 + \|a\|_{L_\infty(0,1)} \right) \int_0^1 x \psi_{1,n}(x)\ \mathrm{d}x + 4 \int_1^\infty x^4 a_n(x) \psi_{1,n}(x)\ \mathrm{d}x \nonumber \\
	& \qquad\qquad \le 4 \left( 1 + \|a\|_{L_\infty(0,1)} \right) + 4 c_2(4)=:c_3 \label{St.7}
\end{align}
for $n\ge 1$. Moreover, a straightforward consequence of \eqref{St.6}, \eqref{St.7}, and H\"older's inequality is that \eqref{St.6} is also true for $m\in (1,3]$ with a suitable constant $c_2(m)$.

We next claim that $(\psi_{1,n})_{n\ge 1}$ is relatively compact in $X_{1,m}$ for any $m\ge 1$. To this end, we note that, thanks to \eqref{St.6}, \eqref{St.7}, \Cref{L.2.1}, and Fubini's theorem,
\begin{equation*}
	\|\psi_{1,n}\|_{X_0} \le \frac{1}{\sqrt{2}} \|\psi_{1,n}''\|_{X_1}^{1/2} \, \|\psi_{1,n}\|_{X_1}^{1/2} \le \sqrt{\frac{c_3}{2}}
\end{equation*}
and
\begin{equation*}
	\|\psi_{1,n}\|_{X_1} + \|\psi_{1,n}'\|_{X_1} \le 1 + \int_0^\infty x \int_x^\infty |\psi_{1,n}''(y)|\ \mathrm{d}y\mathrm{d}x \le 1 + \frac{\|\psi_{1,n}''\|_{X_2}}{2} \le 1+c_2(2)
\end{equation*}
for $n\ge 1$. Now, let $m\ge 2$. In view of \eqref{St.6} and the above estimates, $(\psi_{1,n})_{n\ge 1}$ is a bounded sequence in $X_0\cap X_m\cap W_1^1((0,\infty),x\mathrm{d}x)$ and it follows from \cite[Proposition~7.2.2]{BLL2020b} that $(\psi_{1,n})_{n\ge 1}$ is relatively compact in $X_r$ for any $r\in (0,m)$. As $m\ge 2$ is arbitrary, we conclude that there are $\psi_1\in \bigcap_{r>0} X_r$ and a subsequence $(\psi_{1,n_j})_{j\ge 1}$ of $(\psi_{1,n})_{n\ge 1}$ such that
\begin{equation}
	\lim_{j\to\infty} \|\psi_{1,n_j} - \psi_1\|_{X_r} = 0 \quad\text{ for all }\; r>0\,. \label{St.8}
\end{equation}
An immediate consequence of \eqref{St.8} and the properties of $(\psi_{1,n})_{n\ge 1}$ is that
\begin{equation*}
	\psi_1 \in X_1^+ \;\;\text{ and }\;\; M_1(\psi_1)=1\,.
\end{equation*}
Finally, let $m\ge 1$. We observe that $\dom(\mathbb{A}_{m,n})\subset \dom(\mathbb{A}_m)$ for $n\ge 1$ (as $a_n\ge a$) and that \eqref{St.2} also reads
\begin{equation}
	\mathbb{A}_m \psi_{1,n} = \mathcal{R}_n\,, \label{St.9}
\end{equation}
where
\begin{equation*}
	\mathcal{R}_n(x) := - \frac{x}{n} \psi_{1,n}(x) + \frac{1}{n} \int_x^\infty y b(x,y) \psi_{1,n}(y)\ \mathrm{d}y\,, \qquad x>0\,. 
\end{equation*}
It readily follows from \eqref{St.6} that
\begin{equation*}
	\|\mathcal{R}_n \|_{X_{1,m}} \le \frac{2}{n} \left( \|\psi_{1,n}\|_{X_2} + \|\psi_{1,n}\|_{X_{m+1}} \right) \le \frac{2}{n} \left( c_2(2) + c_2(m+1) \right) \,,
\end{equation*}
so that
\begin{equation}
	\lim_{n\to\infty} \|\mathcal{R}_n \|_{X_{1,m}} =0\,. \label{St.10}
\end{equation}
In view of \eqref{St.9} and \eqref{St.10}, the sequence $(\mathbb{A}_m \psi_{1,n_j})_{j\ge 1}$ converges to zero in $X_{1,m}$ as $j\to\infty$ and, since $\mathbb{A}_m$ is closed on $X_{1,m}$, we readily deduce from \eqref{St.8} that $\psi_1\in \dom(\mathbb{A}_m)$ and $\mathbb{A}_m\psi_1=0$.
\end{proof}

\section*{Acknowledgments}

This work was done while PhL enjoyed the kind hospitality of the Institut f\"ur Angewandte Mathematik, Leibniz Universit\"at Hannover.	

\bibliographystyle{siam}
\bibliography{FragmDiff}
	
\end{document}